\documentclass[11pt]{article}
\usepackage[left=1in, right=1in, top=1in]{geometry}
\usepackage[toc,page]{appendix}
\usepackage[utf8]{inputenc}
\usepackage{amsfonts}
\usepackage{mathrsfs}
\usepackage[english]{babel}
\usepackage{amsmath}
\usepackage{relsize}
\usepackage{lipsum}
\usepackage{xcolor}
\usepackage[english]{babel} 
\usepackage{amsthm}
\usepackage{tikz-cd}
\usepackage{extpfeil}
\usepackage{mathtools}
\usepackage{csquotes}
\usepackage{color}
\usepackage{soul} 
\usepackage{amssymb}
\usepackage[normalem]{ulem}
\usepackage[nottoc]{tocbibind} 
\usepackage{graphicx}
\usepackage{authblk}
\usepackage{amscd}

\usepackage[colorlinks, citecolor=blue]{hyperref} 
\usepackage{bbold}

\newcommand{\mc}[1]{\mathcal{#1}}
\newcommand{\dtnot}{F_{t_0}} 
\newcommand{\xil}{\xi_{\lambda}} 
\newcommand{\rnotm}{R_0[M]}
\newcommand{\dtt}{F_b} 
\newcommand{\maxn}{\mfr^n}
\newcommand{\maxnp}{\mfr^{n+1}}
\newcommand{\vnot}{v_0}
\newcommand{\rmp}{R_0[M']}
\newcommand{\rmpp}{R_0[M'']}
\newcommand{\dvnot}{D_{v_0}} 

\newcommand{\otpr}{\mc{O}_{T'}}
 
\newcommand{\wh}[1]{\widehat{#1}}

\newcommand{\msj}{\mathscr{J}}
\newcommand{\obv}{\mc{O}b_v}  
\newcommand{\henss}{\mc{O}_{S,s}^h}

\newcommand{\frk}[1]{\mathfrak{m}_{#1}}
\newcommand{\on}[1]{\operatorname{#1}}

\newcommand{\imp}{\on{im}(\pi)}

\newcommand{\got}{\Gamma(\ot)} 
\newcommand{\oto}{\mc{O}_{T_0} } 
\newcommand{\ovmn}{\ov{\phi}_{m,n}} 
\newcommand{\goto}{ \Gamma(\mc{O}_{T_0})} 
 
\newcommand{\whov}{\wh{\mc{O}}_{V,v}}
 
\newcommand{\red}{\on{red}}

\newcommand{\ared}{A_{0}} 
 
\newcommand{\ov}[1]{\overline{#1}}

\newcommand{\otr}{\mc{O}_{\terd}} 

\newcommand{\ovv}{\mc{O}_{V,v}}

\newcommand{\z}{\mathbb{Z}}
\newcommand{\bos}{\on{bos}}
\newcommand{\ot}{\mc{O}_T}

\newcommand{\whxx}{\wh{\mxx}}
\newcommand{\omxx}{\ov{\mxx}}

\newcommand{\tura}{\frak{a}} 
\newcommand{\derp}{\on{Der}} 
\newcommand{\muu}{\frak{m}_u}

\newcommand{\mr}{\frk{R}}

\newcommand{\derm}{\on{Der}(R, M)}

\newcommand{\cder}{\mc{D}er} 
\newcommand{\whxi}{\wh{\xi}} 
\newcommand{\rtwid}{\tilde{R}}
\newcommand{\mxx}{\mc{X}}
\newcommand{\sdual}{k[\epsilon, \eta]}
\newcommand{\dual}{k[\epsilon]}
\newcommand{\dodd}{k[\eta]} 
\newcommand{\faka}{\frk{A}}

\newcommand{\jj}{\mc{J}}
\newcommand{\fmb}{\frak{m}_B}

\newcommand{\jar}{\mc{J}_R}
\newcommand{\tant}{t_{\omxx}} 
\newcommand{\plutan}{\tant^+}
\newcommand{\negtan}{\tant^-} 
\newcommand{\tanb}{t_B^*}
\newcommand{\tana}{t_A^*}

\newcommand{\spe}{\on{Spec}}
\newcommand{\mfr}{\mathfrak{m}}
\newcommand{\hr}{\hookrightarrow}
\newcommand{\tas}{t_A^*}
\newcommand{\bas}{t_B^*}

\newcommand{\mm}{\mc{M}}
 
\newcommand{\terd}{T_{\on{red}}} 
\newcommand{\ex}{\on{Ex}} 
\newcommand{\exc}{\mc{E}x} 

\newcommand{\oss}{\mc{O}_{S,s}}
\newcommand{\whos}{\wh{\mc{O}}_{S,s}}
\newcommand{\blam}{B_{\lambda}} 

\newcommand{\lgr}{\overset{\sim}{\longrightarrow}}

\newcommand{\spf}{\on{Spf}}

\newcommand{\sS}{\on{sAffSch}} 
\newtheorem{theorem}{Theorem}[section]
\newtheorem{lemma}[theorem]{Lemma}
\newtheorem{proposition}[theorem]{Proposition}
\newtheorem{corollary}[theorem]{Corollary}
\newtheorem{definition}[theorem]{Definition}

\newtheorem{example}[theorem]{Example}

\newtheorem{remark}[theorem]{Remark}

\usepackage{appendix} 

\title{Artin's theorems in supergeometry} 
\author{Nadia Ott\footnote{\emph{Email address}: \href{mailto:ottnadia@sas.upenn.edu}{ottnadia@sas.upenn.edu}\\ \indent University of Pennsylvania, Philadelphia, PA}}

\begin{document}

\maketitle

\begin{abstract}
We generalize  Artin's three main algebraicity theorems to 
the setting of supergeometry:  approximation \cite{artin1969algebraic},  algebraization of formal deformations \cite{artin1969algebraization}, and algebraization of stacks \cite{artin1974versal}. 
 
\end{abstract}

\setcounter{tocdepth}{1}
\tableofcontents

\section{Introduction}

Artin's theorems on approximation \cite{artin1969algebraic}, 
algebraization of formal
deformations \cite{artin1969algebraization}, and 
stacks \cite{artin1974versal} give general criteria for
functors to be, in various senses, described by algebraic objects.  It has long been expected that 
analogous results hold in supergeometry; for instance, this was indicated
in Deligne's sketch construction of the moduli space of stable super
Riemann surfaces \cite{deligne1987letter}.  
Some works have made progress in this direction 
\cite{vaintrob1990deformation, flenner1992analytic, moosavian2019existence}, and the existence of the moduli space of stable super Riemann surfaces has since been established by other means \cite{moosavian2019existence, felder2020moduli}. 
However, proofs of the full suite of Artin theorems have remained 
absent from the supergeometry literature.\footnote{One reason 
for this absence may be 
the sense that establishing the Artin theorems in supergeometry
would require the tedious repetition of various difficult
arguments in commutative algebra, deformation theory, and algebraic geometry, e.g. for the original approach
\cite{schlessinger1968functors, rim1972formal, 
neron1964modeles, artin1969algebraic, artin1969algebraization, artin1974versal}
(with some of these substituteable by later approaches, 
see e.g. 
\cite{flenner1981kriterium, popescu1986general, 
conrad2002approximation, hall2017openness, hall2019artin}). 
In any case one can find in the literature various assertions
that certain results hold in supergeometry ``by the same argument as''
something in these works.} Here we fill this gap by proving the Artin theorems.

We restrict ourselves
to working over $\spe \mathbb{C}$ 
due to the needs of the subject 
and the comfort of the author. 
We provide complete
proofs, adhering to the principle that we are happy to cite 
{\em statements} in the (bosonic) algebraic geometry literature, 
but not {\em proofs}: if an argument must be repeated with 
the word `super' added, we repeat the argument.\footnote{Usually
with some details added and of 
course with reference and credit given to the original source.}
This is done so that we can verify that the various foundational
results needed to run these arguments are in fact known
in supergeometry (and so
the supergeometer need not be referred to various other works
and verify at every
stage that they still hold with odd coordinates).  Fortunately,
not all arguments are such repetitions: at several
key points we are 
able to reduce to the (known) bosonic case by an argument 
which is significantly simpler 
than the original bosonic argument.  This is 
possible because superschemes are nilpotent thickenings
of bosonic schemes, so the difference between them  does
not involve the difficult `convergence' questions which are often
at the heart of the Artin theorems.

We now describe the results proven in this article. 
Recall that a superscheme 
$X$ gives rise to a functor from 
affine superschemes to sets by 
$X(S) := \on{Hom}(S, X)$; such a functor is called 
{\em representable} and necessarily satisfies \'etale descent.
More generally, a superstack is an
arbitrary 2-functor from the category of 
affine superschemes to the category of groupoids, satisfying \'etale descent. 
(In the body of the text we will as usual
pass to the 1-categorical formulation in terms
of categories fibered in groupoids; 
see Def. \ref{def: superstack}.)

We say a superstack 
is limit preserving if, for a directed
systems of algebras, 
$\varinjlim \mxx( \spe A_i) \xrightarrow{\sim} \mxx (\spe \varinjlim A_i)$. 
As all rings are 
direct limits of its finitely generated sub-rings, 
the `limit preserving' condition ensures 
superstacks are
determined by their values over a finite type superschemes.

We now state the Artin theorems.  
The first -- Artin approximation -- 
asserts that object associated
to the completion of a local
ring of a finite type superscheme
are approximated to any fixed
finite order by some  object
over some finite type 
\'etale cover, depending on the order. 

\begin{theorem}[Artin approximation] \label{thm: intro approximation} Let $S$ be finite type affine superscheme and let $\mxx$ be a limit-preserving superstack. Let $s$ be a closed point in $S$ and let $\wh{\xi} \in \mxx(\wh{\mc{O}}_{S,s})$.  Then for any integer $n \ge 1$, there exists an \'etale morphism $(S', s') \to (S,s): s' \mapsto s$ and an object $\xi \in \mxx(S')$ such that  $\wh{\xi} \vert_{\spe (\mc{O}_{S,s}/\frk{s}^{n})} $ and $ \xi \vert_{\spe(\mc{O}_{S',s'}/\frk{s'}^{n})}$ are isomorphic as objects in $\mxx$ over $\spe (\mc{O}_{S,s}/\frk{s}^{n}) \cong \spe(\mc{O}_{S',s'}/\frk{s'}^{n})$.  
\end{theorem}

The second -- Artin algebraization -- shows
that if we require the initial data $\whxi$ to be \emph{formally versal}, then 
a far stronger conclusion holds: One can
find a single cover and a single algebraic object 
which approximates the original object \emph{up to any order.}  
We recall that an object $\whxi$ is called formally versal if it satisfies the following infinitesimal lifting property: For any surjection of Artin superalgebras, $B' \to B$, and map $R \to B$, there exists an arrow completing the diagram 
\begin{equation} \label{formallyversallift} 
\begin{tikzcd} 
\spe B \arrow[d] \arrow[r]                  & \spe R \arrow[d, "\whxi"] \\
\spe B' \arrow[r] \arrow[ru, dotted] & \mxx                     
\end{tikzcd}
\end{equation}
Note here that for any superscheme $\spe A$ the map  $\spe A \to \mxx$
is by definition a map of functors 
$\hom(\, \cdot \,, \spe A) \to \mxx$, which by
2-Yoneda is identified with an object in $\mxx(\spe A)$. 

\begin{theorem}[Artin algebraization] \label{thm: intro algebraization}  Let $\mxx$ be a limit-preserving superstack and let $(R, \mr)$ be a local, complete, Noetherian superalgebra. Let $\whxi \in \mxx(\spe R)$ be formally versal. Then there exists 
\begin{enumerate}
    \item an affine finite type superscheme $\spe A$ and a closed point $u \in \spe A$, 
    \item an object $\xi_A \in \mxx(\spe A)$, 
    \item an isomorphism $\alpha: R \lgr \wh{A}_{\frk{u}}$ of superalgebras, and
    \item a compatible family of isomorphisms $\wh{\xi} \vert_{\spe(R/\frk{R}^n)} \cong \xi_A \vert_{\spe(A_{\frk{u}}/\frk{u}^n)}$ in $\mxx$ over $R/\mr^{n} \cong A_{\frk{u}}/\frk{u}^{n}$ for all $n \ge 1$. 
\end{enumerate}
\end{theorem}

Formal versality is about what happens over Artin rings and so we can study it using deformation theory. 
When using results from deformation theory we pass to the functor $\omxx_x$ over the category $C$ of local, Artin superalgebras parameterizing isomorphism classes of \emph{extensions} of a fixed object $x \in \mxx(\spe k)$. For a given morphism $\spe k \to S$, we say that $y \in \mxx(S)$ is an extension of $x$ over $S$ if there is an arrow $x \to y$ in $\mxx$ over $\spe k \to S$. 
In general, it not possible to determine formal versality outright; and so we pass to the functor $\wh{\omxx}_x$ sending a complete, local, Noetherian superalgebra $R$ to the set
\[ \wh{\omxx}_x(\spf R) := \varprojlim \omxx_x(\spe R/\mr^n).  \]   
of \emph{formal objects} over the formal superscheme $\spf R$. 

We will now state
Schlessinger's theorem \cite[Theorem 2.11]{schlessinger1968functors} (which predated and motivated Artin's work)  on the existence of formally versal, formal moduli.

. 



\begin{theorem}[Schlessinger's theorem] \label{intro: Schlessinger} Let $A' \to A$ and $A'' \to A$ be morphisms in $C$ and consider the map 
\begin{equation} \label{SM} \omxx_x( \spe A' \bigsqcup_{\spe A} \spe A'') \to \omxx_x( \spe A') \times_{\omxx_x(\spe A)} \omxx_x(\spe A'') \end{equation} 
Then there exists a complete, local, Noetherian superalgebra $R$ and a formally versal, formal object $\whxi \in \wh{\omxx}_x(\on{Spf} R)$ if and only if the following conditions hold.
\begin{itemize}
    \item[\emph{(S1)a:}] \eqref{SM} is a surjection whenever $A'' \to A$ is a small extension. 
    
    \item[\emph{(S1)b:}] \eqref{SM} is a bijection whenever $A=k, A''=\sdual$ . 
    
    \item[\emph{(S2):}] $\omxx_x$ has a finite-dimensional tangent space. 
\end{itemize}
\end{theorem}

The formal object $\whxi$ is a sequence $(x=\whxi_0, \whxi_1, \cdots), \ \whxi_n \in \omxx_x(\spe R/\mr^n)$ of extensions of $x$. It is useful to view a formal object as a  `power series'; the first order term of which is an object in the tangent space to $\omxx_x$. The tangent space to $\omxx_x$
is
$\omxx_x(\mathrm{Spec}(\sdual))$, where $\sdual$ denotes the \emph{super dual numbers}.  
Any finite order extension of $x$ is built from iterating the first order extension and we can use this fact to approximate the extensions of $x$ up to finite order. 
A formal object is what we get when we take this process to the limit. 

This leaves open the important question of convergence. We say that the formal object is {\em effective}
if it is contained in the image of the natural map 
 \begin{equation} \label{effmap} \mxx( \spe R) \to \varprojlim \mxx( \spe R/\frk{R}^n). \end{equation}
 We call the superstack $\mxx$ effective if \eqref{effmap} is an equivalence of categories.
 \begin{remark}
 Effectivity will be a hypothesis in Artin's third theorem on algebraicity of stacks.  Often, effectivity can be checked using the Grothendieck existence theorem; a super version of which is proved in \emph{\cite{moosavian2019existence}}. 
 \end{remark}

While the preceding results have been of a 
local nature, the third Artin theorem is
a global result.  
To state it, let us recall that
an algebraic space is a quotient of a scheme
by an \'etale equivalence relation, and 
that a map of stacks 
$f: \mxx \to \mathcal{Y}$ is said 
to be representable by a scheme
(space) if for any $B \to \mathcal{Y}$, 
one has that $\mxx \times_\mathcal{Y} B$
is a scheme (space).  Properties of maps
of schemes which are stable under 
base change, such as smoothness, 
give rise to properties of representable
morphisms by demanding all base changes
satisfy the original property. 
Representability
of the diagonal $ \Delta: \mxx \to \mxx \times \mxx$ implies that {\em any} map from a scheme
is representable, and so in particular
we may always ask whether
a map from a scheme is smooth, \'etale, etc.

As in \cite{artin1974versal},
we say:

\begin{definition} A superstack $\mxx$ is an algebraic superstack if the following conditions hold: 
\begin{itemize}
    \item[\emph{(1)}] The diagonal morphism
    $ \Delta: \mxx \to \mxx \times \mxx$
    is representable by algebraic superspaces.   
    \item[\emph{(2)}] There exists a smooth, surjective morphism $\pi: U \to \mxx$ with $U$ a superscheme. 
\end{itemize}
\end{definition}

We say that the algebraic superstack $\mxx$ is locally of finite type if there exists a smooth surjection $\pi: U \to \mxx$ with $U$ a  superscheme locally of finite type. 


From   
Schlessinger's theorem  and 
Artin algebraization,
we deduce the (not yet useful) algebraicity criterion:

\begin{corollary} \label{intro: corollary not useful} $\mxx$ is an algebraic superstack locally of finite type if and only if 
\begin{enumerate}
    \item The Schlessinger conditions \emph{(S1)} and \emph{(S2)} are satisfied.
    
    \item Effectivity: If $A$ is a complete, local, Noetherian superalgebra, then the natural map 
    \[  \mxx( \spe A) \to \varprojlim \mxx( \spe A/\frk{A}^n)\] 
   is an equivalence of categories. 
   \item $\mxx$ is limit preserving. 
    \item The diagonal map $\Delta: \mxx \to \mxx \times \mxx$ is representable. 
    \item The maps $\xi_R: (\spe R, u) \to \mxx$ obtained 
    by applying Artin algebraization to each of the formally
    versal, formal objects produced in Schlessinger's theorem are smooth in some neighborhood of $u$.
\end{enumerate}
\end{corollary}
\begin{proof}
Note first that by effectivity, the Schlessinger
formally versal, formal objects are in fact formally 
versal objects.  Thus we are permitted to apply
Artin approximation and get a formally versal object $\xi_R \in \mxx(\spe R)$.  Now 
`if' is obvious: Since we assumed the diagonal
is representable, the $\xi_R$ give us a chart at each
point of $\mxx$, and we have assumed that these 
are all smooth.  

We omit the proof of `only if' as we do not use it, 
but let us just remark about why the maps $\xi_A$ 
must be smooth.  As we assume $\mxx$ is algebraic 
of locally finite type, 
we may study $\xi_A$ after base change 
to some map between finite type schemes.  
By versality, this map
satisfies the lifting condition with respect
to maps from spectra of Artin rings (at $u$).
For maps of schemes of finite type, this is equivalent to smoothness
at $u$. 
\end{proof}

To turn this into a useful criterion, we must 
give  criteria for showing that the formally versal
object $\xi_R$ is in fact smooth. A refinement of this idea is to give criteria which ensure that (1) formal versality is an open condition, and that (2) formal versality and formal smoothness are equivalent notions.

We recall that the object $\xi_R$ is formally versal at a closed point $t \in \spe R$ if it satisfies the following infinitesimal lifting property: For every surjection of Artin superalgebras, $B' \to B$, there exists an arrow completing the diagram  
\begin{equation} \label{nei} 
\begin{tikzcd} 
\spe k(t) \arrow[r, hook] & \spe B \arrow[d] \arrow[r]                  & \spe R \arrow[d, "\xi_R"] \\
{} & \spe B' \arrow[r] \arrow[ru, dotted] & \mxx                     
\end{tikzcd}
\end{equation}
where $k(t)$ is the residue class field of $t$. Formal versality is an open condition if $\xi_R$ formally versal at $t \in \spe R$ implies that there exists a Zariski open subset $ t \in U \subset \spe R$ such that $\xi_R$ is formally versal at every closed point $u \in U$.


Following Flenner \cite{flenner1981kriterium} we deduce criteria for openness of formal versality by considering a certain sheaf $\exc(\xi_R, \mm)$ 
on the superscheme $\spe R$ whose vanishing locus is the set of points at which $\xi_R$ is formally versal.
Observe that if $\exc(\xi_R, \mm)$ were a coherent module, then openness of formal versality would follow automatically from the fact that coherent modules have closed support;  however, in general, $\exc(\xi_R, \mm)$ is just a sheaf of pointed sets. 
Flenner deals with this problem by identifying certain coherent modules on $\spe R$, naturally associated to $\exc(\xi_R, \mm)$, and deducing a set of criteria which ensure that the union of the support of these modules is equal to the support of $\exc(\xi_R, \mm)$. Let us now describe these modules. 

Let $T=\spe R$, and let $\mm$ be a coherent $\otr$-module. Let $\exc(T, \mm)$ be the sheaf on $T$ classifying extensions of $T$ by $\mm$ and let $\mc{D}(\mm)$ be the sheaf on $T$ classifying extensions of $\xi_R$ over trivial extensions of $T$ by $\mm$. The sheaf $\exc(T, \mm)$ is well-known to be coherent, and $\mc{D}(\mm)$ is coherent under the Schlessinger conditions. 
There is a natural map $\pi: \exc(\xi_R, \mm) \to \exc(T, \mm)$ with quotient $\exc(T, \mm)/\pi(\exc(\xi_R, \mm))$ classifying obstructions to extending $\xi_R$. We define $\mc{O}b(\mm)$ to be the quotient $\exc(T, \mm)/\pi(\exc(\xi_R, \mm))$. 

We will now state Flenner's criterion for openness of formal versality (formulated in the super case).

\begin{theorem}[Openness of formal versality] \label{thm: intro openness}  Suppose $\mxx$ is a superstack satisfying the following conditions: 
\begin{enumerate} 
\item Schlessinger's conditions \emph{(S1)} and \emph{(S2)}
\item \emph{Coherence of $\mc{O}b$}: For every affine superscheme $T$, object $v \in \mxx(T)$ and coherent $\otr$-module $\mc{M}$, $\mc{O}b_v(\mc{M})$ is a coherent $\ot$-module.
\item \emph{Constructibility}: Let $v, T$ be as above and suppose that $T' \subseteq T$ is an irreducible reduced subspace of $T$. Then there is a Zariski open dense subset $V'\subseteq T'$ such that for each $t \in V'$, the canonical morphisms 
\begin{itemize}
    \item[\emph{(a)}] $\mc{D}_v(\otpr) \otimes k(t) \to \mc{D}_v(\otpr \otimes k(t))_t$  and $\mc{D}_v(\Pi \otpr) \otimes k(t) \to \mc{D}_v(\Pi \otpr \otimes k(t))_t$ are  bijective, and
    \item[\emph{(b)}] $\mc{O}b_v(\otpr) \otimes k(t) \to \mc{O}b_v(\otpr \otimes k(t))_t$ and $\mc{O}b_v(\Pi \otpr) \otimes k(t) \to \mc{O}b_v(\Pi \otpr \otimes k(t))_t$  are injective. 
\end{itemize}
\end{enumerate} 
Then for every superscheme $T$ and every $v \in \mxx(T)$, the subset of points $t$ in $T$ which are formally versal for $\mxx$ form a Zariski-open subset of $T$. 
\end{theorem}





It is also possible to deduce
formal smoothness from formal versality 
given some additional hypotheses on $\mxx$. 
We recall that the object $\xi_R$ is \emph{formally smooth} if for any infinitesimal extension $B' \to B$ of Noetherian superalgebras and map $\spe B \to \spe R$ there exists an arrow completing the diagram \begin{equation} \label{into: defn formal smoothness}
\begin{tikzcd} 
 \spe B \arrow[d] \arrow[r]                  & \spe R \arrow[d, "\xi_R"] \\
 \spe B' \arrow[r] \arrow[ru, dotted] & \mxx                     
\end{tikzcd}
\end{equation} 


\begin{theorem} \label{thm: intro formal smoothness} Let $\mxx$ be a superstack satisfying the conditions in Theorem \ref{thm: intro openness}. Let $v \in \mxx(T)$ be algebraic and
suppose $\mxx$ satisfies the following additional conditions: 

\begin{enumerate}
    \item  The sheaf $\mc{D}$ is compatible with completion: If $T_0$ is a reduced finite type scheme, $a_0 \in \mxx(T_0)$, $t \in T_0$ and $\mm$ is a coherent $\oto$-module, then \[\mc{D}_{a_0}(\mm) \otimes_{\mc{O}_{T_0}} \wh{\mc{O}}_{T_0,t} \lgr \varprojlim \mc{D}_{a_0}(\mm/\mfr_t^n \mm)_t \]
    
    \item The sheaf $\mc{D}$ is compatible with \'etale localization: If $e: S_0 \to T_0$ is an \'etale morphism of reduced schemes, where $e^*(a_0)=b_0$, and $\mm$ a coherent $\oto$-module, then
\[
    \mc{D}_{b_0}(\mm \otimes_{\oto} \mc{O}_{S_0}) \cong \mc{D}_{a_0}(\mm) \otimes_{\oto} \mc{O}_{S_0}
\]
\end{enumerate}

Then $v$ is formally smooth if and only it is formally versal at every closed point $t \in T$. 

\end{theorem}





Given
the above results on versality and smoothness we can turn 
the non-useful Corollary \ref{intro: corollary not useful} 
into the more checkable criterion which is
the third Artin theorem. 
The conditions (formulated in the super case) are as follows. 

\begin{theorem} \label{mainth}  Let $\mxx$ be a limit preserving superstack and assume that conditions \emph{(S1)} and \emph{(S2)} hold. Then $\mxx$ is a algebraic superstack locally of finite type if and only if
\begin{enumerate}
    \item[\emph{(A1)}.] The diagonal morphism $\Delta: \mxx \to \mxx \times \mxx$ is represented by an algebraic superspace locally of finite type. 
    
    \item[\emph{(A2)}.] \emph{Effectivity}: If $R$ is a complete, local, Noetherian superalgebra, then the natural map 
    \[  \mxx( \spe R) \to \varprojlim \mxx( \spe R/\frk{R}^n)\] 
   is an equivalence of categories.

    \item[\emph{(A3)}.] \emph{Coherence of $\mc{O}b$}: For every superscheme $T$, object $v \in \mxx(T)$ and coherent $\otr$-module $\mc{M}$, $\mc{O}b_v(\mc{M})$ is a coherent $\ot$-module.
    
    
    \item[\emph{(A4)}.] \emph{Constructibility}: Let $v, T$ be as above and suppose that $T' \subseteq T$ is an irreducible reduced subspace of $T$. Then there is a Zariski open dense subset $V'\subseteq T'$ such that for each closed point $t \in V'$, the canonical morphisms 
\begin{itemize}
    \item[\emph{(a)}] $\mc{D}_v(\otpr) \otimes k(t) \to \mc{D}_v(\otpr \otimes k(t))_t$  and $\mc{D}_v(\Pi \otpr) \otimes k(t) \to \mc{D}_v(\Pi \otpr \otimes k(t))_t$ are  bijective, and
    \item[\emph{(b)}] $\mc{O}b_v(\otpr) \otimes k(t) \to \mc{O}b_v(\otpr \otimes k(t))_t$ and $\mc{O}b_v(\Pi \otpr) \otimes k(t) \to \mc{O}b_v(\Pi \otpr \otimes k(t))_t$  are injective. 
\end{itemize}

\item[\emph{(A5)}.] The sheaf $\mc{D}$ is compatible with completion: If $T_0$ is a reduced finite type scheme, $a_0 \in \mxx(T_0)$, $t \in T_0$ and $\mm$ is a coherent $\oto$-module, then \[ \mc{D}_{a_0}(\mm) \otimes_{\mc{O}_{T_0}} \wh{\mc{O}}_{T_0,t} \lgr \varprojlim \mc{D}_{a_0}(\mm/\mfr_t^n \mm)_t \]
    
    \item[\emph{(A6)}.]  The sheaf $\mc{D}$ is compatible with \'etale localization: If $e: S_0 \to T_0$ is an \'etale morphism of reduced schemes, with $e^*(a_0)=b_0$, and $\mm$ is a coherent $\oto$-module, then
\[
    \mc{D}_{b_0}(\mm \otimes_{\oto} \mc{O}_{S_0}) \cong \mc{D}_{a_0}(\mm) \otimes_{\oto} \mc{O}_{S_0}
\]
 \end{enumerate}
\end{theorem}

\begin{proof}  We prove this result
using the statements of the above theorems,
whose proofs we provide in later sections. 

Since (S1) and (S2) hold there exists a formally versal, formal object $\whxi \in \wh{\omxx}_x(\spf R)$ for each $x \in \mxx(\spe k)$ by Schlessinger's Theorem (Theorem \ref{intro: Schlessinger}).  
By (A2), the formal object $\whxi$ can be approximated with a formally versal object $\xi \in \mxx(\spe R)$. We can approximate $\xi$ with a formally versal, algebraic object,  $\xi_R \in \mxx(\spe R)$ by Artin algebraization (Theorem \ref{thm: intro algebraization}). 
The set of closed points in $\spe R$ at which $\xi_R$ is formally versal is Zariski open by (A3) and (A4) and Theorem \ref{thm: intro openness}. Denote this locus by $U \subset \spe R$---possibly shrinking $U$ to a superscheme of finite type---and let $\xi_R$ continue to denote the pullback of $\xi_R$ to an object over $U$. $\xi_R$ is formally smooth by  (A5) and (A6) and Theorem \ref{thm: intro formal smoothness}. The associated map $U \to \mxx$ is representable by (A1). This map is smooth because it is formally smooth and $U$ is of finite type. 
Repeating this process for each $x \in \mxx(\spe k)$, let $\mc{U}= \bigsqcup U_i$ and $u= \sqcup \xi_R$, where the index $i$ is over objects $x \in \mxx(\spe k)$. Then $\mc{U} \to \mxx$ is a smooth, and representable cover of $\mxx$. Furthermore, $\mc{U}$ is locally of finite type since each $U_i$ is of finite type. Thus, $\mxx$ is an algebraic superstack locally of finite type. 
\end{proof}

\vspace{2mm} 
{\bf Strategy and outline.}
Artin approximation (Theorem
\ref{thm: intro approximation} above) is
proven as Theorem \ref{superArt} below; 
we follow \cite{conrad2002approximation}
using the observation of 
\cite{moosavian2019existence} that
the (deep and difficult) N\'eron-Popescu
theorem \cite{neron1964modeles, popescu1986general}
can be extended to super commutative algebra
by a simple argument.  

 Artin algebraization (Theorem \ref{thm: intro algebraization} above) is proven as 
Theorem \ref{superartalg} below.  We adapt some
ideas from \cite{conrad2002approximation},
but are able to avoid many difficult convergence
arguments by reducing to the bosonic case.  

The proof of Schlessinger's theorem (Theorem \ref{intro: Schlessinger} above) is identical
to the bosonic case; we provide it in 
Appendix \ref{schlessinger proof appendix}
 for completeness. 

Openness of versality (Theorem \ref{thm: intro openness} above) is proven as Theorem \ref{flenneropen} below.  We follow \cite{flenner1981kriterium}, except that the proof of the key
lemma about half-exact functors (Lemma \ref{copr}, 
the counterpart of \cite[Theorem 6.1]{flenner1981kriterium}) is done
by reduction to the bosonic case.\footnote{ 
It is asserted in 
\cite{flenner1992analytic} that
Flenner's results on openness of versality 
generalize to the superanalytic setting, 
but few proofs are explicitly written; 
missing in particular is any argument that
the nontrivial   \cite[Lemma 4.1]{flenner1981kriterium} holds 
in supercommutative algebra.} 

We show in Theorem \ref{versality} that
(under appropriate hypotheses), 
formal versality in fact implies versality; 
here we begin by following Artin's original proof \cite[Theorem 3.3]{artin1974versal}, but then avoid a convergence question by reducing to the bosonic case. 

Finally, the relationship between formal 
versality and formal smoothness (Theorem \ref{thm: intro formal smoothness} above) is proven as 
Theorem \ref{formsmooth} below, 
following Artin's original proof \cite[Proposition 4.3]{artin1974versal}.

\vspace{2mm}

{\bf Acknowledgements.}  
I would like to thank Jack Hall for many helpful explanations. I would also like to thank Vivek Shende, Ron Donagi, and Alexander Voronov for helpful discussions. 

\section{Notations and conventions} 

The letter $k$ will denote an algebraically closed field of characteristic zero. We consider the following categories. 
\begin{quote}
    $\on{sAffSch}$= category of Noetherian affine superschemes over $ \spe k$ whose residue fields at closed points are all equal to $k$.  
   
\end{quote}

\begin{quote} $C$=category of local, Artin superalgebras over $k$ with residue field at the closed point equal to $k$. 
    
\end{quote}

\begin{quote} $C^*$=category of local, Noetherian, Henselian superalgebras over $k$ with residue field at the closed point equal to $k$. 
    
\end{quote}
All super vector spaces are assumed to be over $k$.  Morphisms of super vector spaces are  grading-preserving. 
All superalgebras are over $k$, and are assumed to be Noetherian and supercommutative. Morphisms of local rings are always local.  

We denote the $\z_2$-grading on a superalgebra $A$ by $A=A^+ \oplus A^-$ and write $\mc{J}_A$ for the ideal generated by $A^-$ in $A$. For a homogeneous element $a \in A$, $|a|=i \ \on{mod} 2$ will denote the parity of $a$. We say that $a$ is even (resp. odd) if $|a|=0$ (resp. $|a|=1$). We write $A_{\bos}$ for the ordinary $k$-algebra $A/\mc{J}_A$ and $A_{\red}$ for the quotient of $A$ by the ideal of \emph{all} nilpotents in $A$. Given a superscheme $S$, we denote the $\z_2$-graded structure sheaf of $S$ by $\mc{O}_S= \mc{O}_{S}^+ \oplus \mc{O}_{S}^-$ and write $\mc{J}_S = \langle \mc{O}_{S}^- \rangle$ for the sheaf of odd nilpotents in $S$. 


\section{Preliminaries} \label{schlessinger}

\subsection{Super algebraic geometry}

A \emph{super vector space} over $k$ is a vector space $V$ with a $\z_2$-grading, $V=V^+ \oplus V^-$ where $V^+$ and $ V^-$ are  $k$-vector spaces in the ordinary sense. The super dimension of $V$, denoted by $\on{sdim}_k(V)$, is a pair of non-negative integers $r|s$ where $r = \on{dim}_k V^+$ and $s = \on{dim}_k V^-$. We say that $V$ is finite dimensional if $r < \infty$ and $ s < \infty$. 

A \emph{superalgebra} over $k$ is a super vector space which is an associative algebra with unit $1$ and such that the multiplication $A \otimes A \to A$ is a morphism of super vector spaces. A superalgebra is supercommutative if for all $a,b \in A$ we have that 
\[ ab = (-1)^{|a||b|} ba.  \] 
A superalgebra is Noetherian if it satisfies the ascending chain condition on ideals. 

Henceforth, the word superalgebra will mean a Noetherian, supercommutative $k$-superalgebra.

A super module $M$ over a superalgebra $A$ is a super vector space on which $A$ acts from the left, where the action $a \otimes m \mapsto a \cdot m, \ \ (a \in A, m \in M)$ is a morphism of super vector spaces, that is, $|a \cdot m| = |a| + |m|$.  
Henceforth, we write $A$-module to mean $A$ super module. 

Every superalgebra $A$ has an ideal $\mc{J}$ of odd nilpotents generated by $A^-$. $\mc{J}$ has a canonical filtration:
\[  \mathscr{J}: \jj \supset \jj^2 \supset \jj^3 \supset \cdots \supset \jj^r, \] 
where $r$ is the smallest integer for which $\jj^{r+1}=0$. The \emph{$\jj$-associated graded of $A$} is the $\z \times \z_2$-graded $A/\jj$-module
\[ \on{gr}_{\jj}(A) := A/\jj \oplus \jj/\jj^2 \oplus \cdots \oplus \jj^r \]
.

A \emph{superspace} is a \emph{locally super ringed space} $(X, \mc{O}_X)$ consisting of a topological space $X$ and a structure sheaf 
 \[ \mc{O}_X = \mc{O}_X^+ \oplus \mc{O}_X^-\]
 of superalgebra such that the stalks are local superrings. 
 
A \emph{superscheme} $(X, \mc{O}_X)$ is a superspace $(X, \mc{O}_X)$ such that $(X, \mc{O}_X^+)$ is an ordinary scheme and $\mc{O}_X^-$ is a quasi-coherent sheaf of $\mc{O}_X^+$-modules. A \emph{morphism of superschemes} $f:(X, \mc{O}_X) \to (Y, \mc{O}_Y)$ is a continuous map $f:X \to Y$ of topological spaces along with a morphism $f^{\#}: \mc{O}_Y \to f_* \mc{O}_X$ of sheaves of superalgebras inducing a local morphism on stalks. 


A superscheme $X$ is affine if it is the $\spe$ of a superalgebra $A$. An affine superscheme $\spe A$ is of finite type, Noetherian, etc. if the superalgebra $A$ is is finite type, Noetherian, etc. 

\begin{definition}[Flat morphism] A morphism $f: X \to Y$ of superschemes is \emph{flat} if $f^{-1} \mc{O}_Y$ is a flat $\mc{O}_X$-module. 
 
\end{definition}

 \begin{definition}[Smooth morphism] \label{smooth} A morphism $f: X \to Y$ of superschemes is \emph{smooth} at $x \in X$ if the following hold:
 
 \begin{itemize} 
 \item[(a)] $f$ is of finite type at $x$.
 \item[(b)] $f$ is flat at $x$.
 \item[(c)] If $y=f(x)$, then $X_y = X \times_Y \on{Spec} k(y)$ is regular at $x$. 
 
 \end{itemize}
 We say that $f$ is \emph{smooth of relative dimension $(m|n)$} if $f$ is smooth and for each $y=f(x)$, $\on{sdim}_{k(y)} X_y = (m|n)$.

 \end{definition} 
 
  \begin{definition}[\'Etale morphism] \label{etalemorph} A smooth morphism $f: X \to Y$ of superschemes is $\acute{e}$tale if for every $y=f(x)$, $\on{sdim}_{k(y)} X_y = (0|0)$. We say that $f:X \to Y$ is an $\acute{e}$tale covering of $Y$ if $f$ is also surjective. 
 
 \end{definition} 


\begin{lemma}
\label{locallyetalesubscheme} Let $f: X \to Y$ be a smooth morphism of superschemes of relative dimension $(m|n)$ and let $y \in Y$ be a closed point of $Y$, let $x \in X_y$ be a point in the fiber over $y$. Then there exists a Zariski open subset $x \in U \subseteq X$ and an \'etale morphism of $Y$-superschemes $U \to \mathbb{A}_Y^{m|n}$ such that the following diagram  commutes
\begin{center}
    \begin{tikzcd}
    U \arrow[rd, swap, "f \vert_U"] \arrow[r] & \mathbb{A}_Y^{m|n} \arrow[d, "\on{pr}"] \\
    {} & Y 
    \end{tikzcd}
\end{center}

\end{lemma} 

\subsection{Extensions and deformations}

An \emph{infinitesimal extension} $(R', \phi)$ of a superalgebra $R$ is a surjection $\phi: R' \to R$ of superalgebras such that $\on{ker}(R' \to R)=I$ is nilpotent. An infinitesimal extension of a superscheme $T$ is a closed immersion $i: T \hr T'$ of superschemes defined by a sheaf of ideals $\mc{I} \subset \mc{O}_{T'}$ such that $\mc{I}^n =0 $ for $n >1$.  

An morphism of infinitesimal extensions $(R', \phi) \to (R'', \phi')$ of $R$ is a commutative diagram
\begin{equation}
 \begin{tikzcd}
 0 \arrow[r] & I \arrow[r] \arrow[d, double, no head] & R' \arrow[r, "\phi"] \arrow[d, " \alpha"] & R \arrow[r] \arrow[d, double, no head] & 0 \\
 0 \arrow[r] & I \arrow[r] & R'' \arrow[r, "\phi''"] & R \arrow[r] & 0. 
 \end{tikzcd}
\end{equation}
The map $\alpha$ is an isomorphism by the Snake lemma. It follows that the extensions $(R', \phi)$ of $R$ by $I$ form a groupoid. 

As in the classical case, any infinitesimal extension $(R', \phi)$ of $R$ can be factored into a finite composite of square-zero extensions, the ideal $I$ of which is naturally a finitely generated $R$-module. 
  

Let $I$ be a finitely-generated $R$-module and let $\ex(R, I)$ denote the set of isomorphism classes of square-zero extensions of $R$ by $I$. 
This gives a functor $\ex(R,-): \on{Coh}(R) \to \on{Set}$ sending a finitely generated $R$-module $I$ to the set $\ex(R,I)$. There is an $R$-module structure on $\ex(R,I)$ which is described exactly as in the classical case, (see e.g, \cite{sernesi2007deformations}): In particular,  $\ex(R,-): \on{Coh}(R) \to \on{Coh}(R)$.

A square-zero extension $(R', \phi)$ of $R$ is called \emph{trivial} if it has a section $s: R \to R'$ such that $\phi s=1_R$.
Given an $R$-module $I$ we can always construct a trivial extension of $R$ by $I$: Define on $ R \oplus I$ the multiplication
\begin{equation} \label{multtriv} (r, i)  \cdot (r', i') = (rr', r'i + i' r), \ \ r, r' \in R, \ i, i' \in I. 
\end{equation} 
Observe that $I^2=0$. Henceforth, we set $R[I]:=(R \oplus I, \cdot)$. We identify $R[I]$ with the zero element in $\ex(R,I)$. 




\begin{lemma} \label{torsorder} Let $(R', \phi)$ be a square-zero extension of $R$ and let $f: B \to R$ be a morphism of superalgebras. Let  $f_1, f_2: B \to R'$ be two superalgebra morphisms such that the following diagram commutes. 
     \begin{equation} \label{torsordig} 
\begin{tikzcd}
            &             & B \arrow[d, shift left, "f_1"] \arrow[d, shift right, swap, "f_2"] \arrow[rd, "f"] &             &   \\
0 \arrow[r] & I \arrow[r] & R' \arrow[r, "\phi"]                        & R \arrow[r] & 0
\end{tikzcd}
    \end{equation}
Then the induced map $f_1 - f_2: B \to I$ is a grading-preserving $k$-linear derivation.

\end{lemma}

The maps $f_1$ and $f_2$ in Lemma \ref{torsorder}  are called \emph{lifts} of $f$. In particular, Lemma \ref{torsorder} shows that the set of lifts of $f$ form a $\on{Der}_k(B,I)$-torsor. In the case that  $B=R$ and $f=1_R$, giving a lift of $1_R$ is equivalent to giving a section $s: R \to R'$ such that $\phi s=1_r$.  


\subsubsection{Artin superalgebras}

Let $C$ denote the category of local, Artin superalgebras over $k$ with residue field at the unique closed point equal to $k$. 

\begin{example} \normalfont An important example of a local, Artin superalgebra are the super dual numbers $k[\epsilon, \eta]/(\epsilon^2, \epsilon \eta), \ |\epsilon|=0, \ |\eta|=1$. 

\end{example}

As in the ordinary case, any surjection $A' \to A, \ \on{ker}(A' \to A)=I$ in $C$ is nilpotent. 
We say that a square-zero extension $A' \to A, \ \on{ker}(A' \to A)=I$ with $I=(t_1, \cdots, t_r, \theta_1, \cdots, \theta_s), \ |t_i|=0, \ |\theta_i|=1$ is $(r|s)$-small if $\mfr_{A'} \cdot I=0$ so that $I$ is naturally an $(r|s)$-dimensional super vector space over $k=A'/\mfr_{A'}$. 
 
 \begin{example} \normalfont The surjection $k[\eta] \to k, \ |\eta|=1, \ \on{ker}(k[\eta] \to k)=(\eta)$ is a $(0|1)$-small extension of $k$ by $(\eta) \cong \Pi k$. 
The surjection $\sdual \to k, \ \on{ker}(\sdual \to k)=(\epsilon, \eta)$ is a $(1|1)$-small extension of $k$ by $(\epsilon, \eta) \cong k \times \Pi k$.

 \end{example}
 
 \begin{lemma} \label{factorization} Let $f: A' \to A, \ \on{ker}(A' \to A)=I$ be a square-zero extension in $C$. Then $f$ can be factored as a composite of $(1|0)$- and $(0|1)$-small extensions. 
 
 \end{lemma}
 
\begin{proof} The maximal ideal $\mfr_{A'}$ is nilpotent since $A'$ is Artin and so there exists some integer $m \ge 0$ such that $\mfr_{A'}^m=0$. From this we get a factorization,  
 \[ A' = A'/I \mfr_{A'}^{m-1} \to A'/I\mfr_{A'}^{m-2} \to \cdots \to A'/I =A. \] 
of surjective maps whose kernels $ \mfr_{A'}^iI/\mfr_{A'}^{i+1} I$ are all annihilated by $\mfr_{A'}$. That is, $f$ can be factored as a sequence of small $(r_i|s_i)$-extensions, for $i=1, \cdots, m-1$, where $(r_i|s_i) = \on{sdim}_k(\mfr_{A'}^iI/\mfr_{A'}^{i+1} I)$.  Thus it suffices to prove the lemma when $f$ is a $(r|s)$-small extension. Take a basis $(t_1, \dots, t_r, \theta_1, \dots, \theta_s)$ for $\on{ker}(f)=I$ as a super vector space over $k$. From this choice we get a  factorization, 
\[ A' \to A'/(t_1) \to A'/(t_1, t_2) \to \cdots \to A'/(t_1, \dots, t_r) \to A'/(t_1, \dots, t_r, \theta_1) \to \cdots \to A'/(t_1, \dots, \theta_s)=A \]
into $s$ $(0|1)$-small extensions and  $r$ $(1|0)$-small extensions.


\end{proof} 




 \begin{definition} An infinitesimal deformation of a superscheme $Y$ is a cartesian diagram of morphisms of superschemes 
 \begin{equation} \label{defn: deformation} 
     \begin{tikzcd}
     Y \arrow[r] \arrow[d] & \mc{Y} \arrow[d, "\pi"] \\
    \spe k \arrow[r, "s"] & S 
     \end{tikzcd}
 \end{equation}
 with $S$ an Artin superscheme, and $\pi$ flat and surjective. 
 \end{definition}

\subsection{Categories fibered in groupoids} 

A \emph{category fibered in groupoids} over the category $\sS$ is a functor $p: \mxx \to \sS$ with the following properties: 
\begin{enumerate}
\item For every morphism $f: S \to T$ in $\sS$ and for every object $t$ in $\mxx$ over $T$, there exists an object $s \in \mxx(S)$ and a cartesian arrow $\ov{f}:s \to t$.

\item Compositions of cartesian arrows are cartesian. 

\item $\ov{f}: s \to t$ is an isomorphism if $p(\ov{f})$ is. 
\end{enumerate}

The fiber $\mxx(T)$ is the groupoid whose objects are the objects in $\mxx$ over $T$ and whose morphisms are arrows in $\mxx$ which go to $\on{id}_T$ under $p$.  
Henceforth, we will refer to $\mxx$ as simply a groupoid. 



\begin{example} \normalfont  
The functor $\on{Def}_Y: C \to \on{Groupoid}$ sending an Artin superscheme $S$ to the groupoid of infinitesimal deformations of $Y$ over $S$ is a category fibered in groupoids. To check this, 
let $f: S \to T$ be a morphism in $C$ and let $\mc{Y}$ be a deformation of $Y$ over $T$. Then the top horizontal arrow in the fiber product  diagram 
\begin{center}
    \begin{tikzcd}
    \mc{Y} \times_T S \arrow[r] \arrow[d] & \mc{Y} \arrow[d] \\
    S \arrow[r, "f"] & T
    \end{tikzcd}
\end{center}
is a cartesian arrow $\mc{Y} \times_T S \to \mc{Y}$ over $f: S \to T$. This proves condition 1. Conditions 2  and 3 are now obvious from properties of the fiber product.

\end{example}

For a fixed object $v \in \mxx(T)$ and a morphism of superschemes $f: T \to S$, we write $\mxx_v(S)$ for the groupoid of maps $v \to s$ in $\mxx$ lying over $f$. An isomorphism $(v \to s_1) \to (v \to s_2)$ in $\mxx_v$ is a commutative diagram 
\begin{center}
\begin{tikzcd} \label{isosgr}
                        & s_1 \arrow[dd, "\alpha"] \\
v \arrow[rd] \arrow[ru] &                          \\
                        & s_2                     
\end{tikzcd}
\end{center}



A bar will denote the set of isomorphism classes in $\mxx$. For example,  $\ov{\mxx}_v(S)$ is the set of isomorphism classes in the groupoid $\mxx_v(S)$. 

\begin{example} \normalfont  $\ov{\on{Def}}_Y: C \to \on{Set}$ is a functor sending an Artin superscheme $S$ to the \emph{set} of isomorphism classes of deformations of $Y$ over $S$.  

\end{example}

\begin{definition} \label{def: limit preserving} We say that a groupoid $\mxx$ \emph{limit-preserving} if for every directed system $\{ A_i \}_i$ of superalgebras the natural map
\[ \varinjlim \mxx( \spe A_i) \to \mxx (\spe \varinjlim A_i)  \]
is an equivalence of categories. \footnote{It is customary to call this property limit-preserving instead of the more accurate colimit-preserving.  } 
\end{definition}

The following definitions of a superstack and an algebraic superstack are straightforward generalizations of their ordinary counterparts.

\begin{definition}[Superstack]
\label{def: superstack}
A groupoid $\mxx$ is a \emph{superstack} if it satisfies the following conditions: 

\begin{enumerate} 
\item For every pair $x,y \in \mxx(T)$, the isomorphism functor
\[ \on{Isom}(x,y): \sS/T \to \on{Set}: (f: S \to T) \longmapsto \{ \phi: f^*x \lgr f^*y \  |  \ p(\phi)= \on{id}_T \} \] 
is a sheaf in the \'etale topology on $T$. 
\item For any \'etale covering $\{ A \overset{\phi_i}{\longrightarrow} B_i \}$, every descent datum for $\mxx$ relative to $\phi_i$ is effective. 
\end{enumerate} 
\end{definition}

\section{Schlessinger's conditions} 

In this section we provide super versions of Artin's generalization \cite{artin1974versal} of Schlessinger's conditions \cite{schlessinger1968functors}.  

To generalize Schlessinger's conditions,  we consider arbitrary maps of infinitesimal extensions of a given reduced algebra $R_0$ and, in particular, diagrams
\begin{equation} \label{infextr} R' \to R \to R_0, \ \ \on{ker}(R' \to R)= M \end{equation}
 of infinitesimal extensions of $R_0$, where $R' \to R$ is surjective and $M$ is naturally a finite  $R_0$-module.

\begin{remark} \normalfont Schlessinger's original conditions \ref{intro: Schlessinger} are for extensions of $k$. Artin's conditions, as we will see, are formulated for extensions of an arbitrary reduced ring $R_0$. 
The condition that $M$ be a finite $R_0$-module is coming from the fact that any surjection $f: R \to R_0$ can be factored into a finite composite of extensions by finite $R_0$-modules. This is a generalization of the fact that any surjection of local, Artin superalgebras can be factored into a finite composite of extensions by finite dimensional super vector spaces (Lemma \ref{factorization}).








 
\end{remark} 
 


%
In terms of superschemes, we consider arbitrary maps of infinitesimal extensions of a given reduced scheme $T_0$, and, in particular, diagrams 
 \begin{equation} \label{infext} T_0 \hookrightarrow T \hookrightarrow T', \ \on{ker}(i^{-1}\mc{O}_{T'} \to \mc{O}_T)= \mm \end{equation}
 of infinitesimal extensions of $T_0$, where $T \hr T'$ is a closed immersion and $\mm$ is a coherent $\mc{O}_{T_0}$-module.

The following conditions are super versions of Artin's conditions (S1) and (S2) in \cite{artin1974versal}. 
\newline

\textbf{Condition.} (S1). (a)  Let $\mxx$ be a superstack over $\sS$. Let \begin{equation} \label{sonea}
\begin{tikzcd} 
T_0 \arrow[r, hook] & T \arrow[r] \arrow[d, hook] & Y \\
                        & T'                                              &  
\end{tikzcd}
\end{equation} 
be a diagram in $\sS$, where $T_0 \hr T \to T'$ is as in \eqref{infext}.  Assume that the composition $T_0 \hr T \to Y$ is a closed immersion. Let $t \in \mxx(T)$. Then the natural map 
\begin{equation} \label{semihomogeneity} 
\ov{\mxx}_t(T' \sqcup_T Y) \to \ov{\mxx}_t(T') \times \ov{\mxx}_t(Y),
\end{equation}
is surjective.  
\newline

\textbf{Condition.} (S1). (b) Let $i: T_0 \hr Y$ be a closed immersion, $T_0$ a reduced scheme, and let $\mm$ be a coherent $\mc{O}_{T_0}$-module. Let $y \in \mxx(Y)$ such that $t_0 \to y$ is an arrow in $\mxx$ over $i$. Then the natural map 
\begin{equation}  \label{sonebe} \ov{\mxx}_{y}(Y[\mm]) \to \ov{\mxx}_{t_0}(T_0[\mm])) \end{equation}
is bijective. 
\newline

Set 
\begin{equation} \label{defnd}
\omxx_{t_0}(T_0[\mm]) = D_{t_0}(\mm)
\end{equation}
\newline

We check carefully in Lemma \ref{modulestructure} that (S1)b gives $D_{t_0}(\mm)$ a natural $\goto$-module structure, and that when (S1)a holds the additive group $D_{t_0}(\mm)$ acts transitively on the set $\omxx_t(T')$.
\newline 

\textbf{Condition.} (S2) $D_{t_0}(\mm)$ is a coherent $\goto$-module. 

\vspace{2mm} 

Conditions (S1) and (S2) have the obvious reformulations in terms of superalgebras, \emph{i.e.}, replace $\mm$ with its associated $R_0$ module $M$, $T_0 \hr T$ with the associated surjection $R \to R_0$, etc.

\begin{lemma}  \label{modulestructure} \emph{(S1)b} gives $D_{t_0}(\mm)$ a natural $\goto$-module structure, and when \emph{(S1)a} holds the additive group $D_{t_0}(\mm)$ acts transitively on the set $\omxx_t(T')$.  
\end{lemma} 

\begin{proof} 


Let $\phi: B \to R_0$, $M$ and $t_0$ be as in the hypotheses of (S1)b 
and set $F_{t_0}(-)=\omxx_{t_0}^{op}(-)$. 
Our first goal is to prove that the natural map
\begin{equation} \label{original}
    \dtnot(B \times_{R_0} \rnotm) \to F_{t_0}(B) \times \dtnot(\rnotm). 
\end{equation}
is bijective by (S1)b. 

First observe that the map $B \times_{R_0} \rnotm \to B[M]: (b, \phi(b) + r m) \mapsto (b, r m), \ (b \in B, r \in R_0, m \in M)$ is an isomorphism of superalgebras.

For each $b \in F_{t_0}(B)$, the natural map 
\begin{equation} \label{schlesslemms} \dtt(B \times_{R_0} \rnotm) \to \dtt(B) \times \dtnot(\rnotm) \end{equation}
is a bijection since $\dtt(B)=b$ and since $\dtt(B \times_{R_0} \rnotm)=\dtt(B[M])$ is equal to $\dtnot(\rnotm)$ by (S1)b. This proves that  \eqref{original} is a bijection. 


We will now define a $R_0$-module structure on $D_{t_0}(M)=\dtnot(\rnotm \times_{R_0} \rnotm)$ using the bijection \eqref{original} with $B=R_0[M]$.

For the additive structure: For each $t \in D_{t_0}(M)$, there is an arrow $s^*(t) \to t$ in $D_{t_0}(M)$ over \[ s: R_0[M] \times_{R_0} R_0[M] \to R_0[M]: ((r,m), (r,m')) \mapsto (r, m+m'),  \ \ r,r' \in R, \ m,m \in M.  \] 
In particular, we have a map $+: D_{t_0}(M \times M) \to D_{t_0}(M): t \to s^*(t)$. Since $D_{t_0}(M \times M)=D_{t_0}(M) \times D_{t_0}(M)$ by (S1)b, the map $+$ defines an additive structure on $D_{t_0}(M)$. 
Multiplication by $c \in R_0$ is induced by the morphism 
 \[ R_0[M] \to R_0[M]: (a,m) \mapsto (a, cm). \] 
The additive identity (or zero object) in $D_{t_0}(M)$ is just the trivial extension of $t_0$ over $\rnotm$.  



The proof that $D_{t_0}(M)$ acts transitively on $\omxx_t(T')$ can be proved exactly as in \cite[Remark 2.17]{schlessinger1968functors}. 

\end{proof}

Schlessinger's theorem
\cite[Theorem 2.11]{schlessinger1968functors} 
asserts that (S1) and (S2) are enough to
ensure the existence of formally versal, formal objects for $\mxx$.

\begin{theorem} \label{superschless} Suppose\emph{ (S1)} and \emph{(S2)} hold for $\omxx_x$. Then there exists a complete, local, Noetherian superalgebra $R$ and a formally versal, formal object $\whxi \in \wh{\omxx}_x(\on{Spf} R)$. 
\end{theorem}

We check carefully in Appendix 
\ref{schlessinger proof appendix} that this result can
be established for superalgebras by the same argument as
in the ordinary case, with care to take account also for
odd directions.

\section{Artin approximation} \label{artapp}

Following \cite{conrad2002approximation}, 
we will deduce Artin approximation from 
the Popescu desingularization theorem.  
The latter is a very difficult result in 
commutative algebra.  Fortunately, as observed
by \cite{moosavian2019existence}, the supercommutative
version of Popescu is easily reduced to the 
original theorem.  We begin by recalling this result:

\begin{lemma}[{\cite[Lemma 7.9]{moosavian2019existence}}] Let $X$ and $Y$ be affine superschemes and let $f: X \to Y$ be a flat morphism such that $f^{-1}(\jj_Y)$ generates $\jj_X$. Then the commutative diagram 
\begin{center}
    \begin{tikzcd}
    X \arrow[r, "f"] \arrow[d] & Y \arrow[d] \\
    X_{\on{ev}} \arrow[r, "f_{\on{ev}}"] & Y_{\on{ev}} 
    \end{tikzcd}
\end{center}
is cartesian, and $f_{\on{ev}}$ is flat. 
\end{lemma}

\begin{theorem}[\cite{moosavian2019existence}, Lemma 7.30] \label{pop} Let $(A, \mfr)$ be a local Noetherian superring and suppose that the natural morphism $A^+ \to \wh{A}^+$ is geometrically regular. 

Then, 
\[ \wh{A} = \varinjlim_{\lambda} C_{\lambda}\] 
where $C_{\lambda}$ are smooth $A$-superalgebras. 
\end{theorem} 
\begin{proof} 
Here $\wh{A}^+$ denotes the completion of $A^+$ with respect to the even component $\mfr^+$ of the maximal ideal $\mfr$.  Since $A^+ \to \wh{A}^+$ is geometrically regular, we can use the ordinary Neron-Popescu theorem to conclude that
\[ \wh{A}^+= \varinjlim_{\lambda} C_{\lambda} \] 
with $C_{\lambda}$ smooth $A^+$ superalgebras. Let  $f$ denote the natural map $ A \to \wh{A}$.  $f$ is  well-known to be flat. Furthermore,  it is easy to see that  $f^{-1}(\mc{J}_{\wh{A}})$ generates  $\mc{J}_A$. Therefore, we can apply \cite[Lemma 7.9]{moosavian2019existence} 
to conclude that the diagram
\begin{center}
\begin{tikzcd}
A^+ \arrow[r, "f_0"] \arrow[d] & \wh{A}^+ =\varinjlim_{\lambda} C_{\lambda} \arrow[d] \\
A \arrow[r, "f"]               & \wh{A}            
\end{tikzcd}
\end{center}
is cocartesian. Thus, $\wh{A} = \varinjlim_{\lambda}(C_{\lambda}) \otimes_{A^+} A = \varinjlim_{\lambda} (C_{\lambda} \otimes_{A^+} A)$. The result now follows since $C_{\lambda} \otimes_{A^+} A$ is smooth over $A^+$ and so $C_{\lambda} \otimes_{A^+} A$  is smooth over $A$. 


\end{proof}

\begin{theorem}[Artin approximation] \label{superArt} Let $S$ be a finite type affine superscheme and let $\mxx$ be a limit-preserving superstack. Let $s$ be a closed point in $S$ and let $\wh{\xi} \in \mxx(\spe \wh{\mc{O}}_{S,s})$.  Then for any integer $n \ge 1$, there exists an \'etale morphism $(S', s') \to (S,s): s' \mapsto s$ and an object $\xi \in \mxx(S')$ such that  $\wh{\xi} \vert_{\spe (\mc{O}_{S,s}/\frk{s}^{n})} $ and $ \xi \vert_{\spe(\mc{O}_{S',s'}/\frk{s'}^{n})}$ are isomorphic as objects in $\mxx$ over $\spe (\mc{O}_{S,s}/\frk{s}^{n}) \cong \spe(\mc{O}_{S',s'}/\frk{s'}^{n})$.  
\end{theorem}

\begin{proof}
We follow 
\cite{conrad2002approximation}, drawing also
from the exposition \cite{alper2015artin}. Since $S$ if of finite type, the natural map $\mc{O}_{S,s}^+ \to \wh{\mc{O}}_{S,s}^+$ is geometrically regular, \cite[7.4.4]{grothendieck1965elements}. 
Thus,
\[ \whos = \varinjlim B_{\lambda}  \] 
where $B_{\lambda}$ are smooth $\oss$-superalgebras, by Theorem \ref{pop}. 
Since $\mxx$ is limit-preserving, there exist some large enough $\lambda$, and an object $\xi_{\lambda} \in \mxx(\spe \blam) $ such that $\whxi \to \xi_{\lambda}$ is an arrow over the natural morphism $\phi: \spe \whos \to \spe \blam$ of superschemes over $\spe \oss$. 

Using this fact, we will now prove that for each $n \ge 1$, there exists a pointed superscheme $(S',s')$ and an object $\xi \in \mxx(S')$ with the following properties: 
\begin{enumerate} 
\item there is a closed immersion
$i: S' \hr \spe \blam:s' \to s$, 

\item the composition $S' \hr \spe \blam \to \spe \oss$ is \'etale, and 

\item $\xi=i^*(\xil) \vert_{\spe(\mc{O}_{S',s'}/\mfr_{s'}^n)} $ and $ \whxi \vert_{\spe(\whos/\mfr_s^n)}$ are isomorphic as objects in $\mxx$ over $\spe(\mc{O}_{S',s'}/\mfr_{s'}^n) \cong \spe(\whos/\mfr_s^n) \cong \spe (\oss/ \mfr_s^n)$. 

\end{enumerate} 

The remainder of the proof concerns the construction 
of $S'$.
\newline

For (1): Since $\blam$ is a smooth $\oss$-superalgebra, $\Omega_{\blam/\oss}^1$ is a locally free $\blam$-module. Let $ s \in U \subset \spe \blam$ be an affine Zariski open neighborhood $s$ on which $\Omega_{\blam/\oss}^1$ trivializes. Choose a basis $du_1, \dots, du_n, d \mu_1, \dots, d \mu_m, \ |du_i|=0, |d \mu_i|=1$ for $\Omega_{\blam/\oss}^1 \vert_U$. The map 
\begin{equation} \label{etalething} \oss[x_1, \dots, x_n, \eta_1, \dots, \eta_m] \to \blam: x_i \mapsto u_i, \eta_i \mapsto \mu_i \end{equation} 
is an \'etale morphism of smooth $\oss$-superalgebras.

Let $E=\oss[x_1, \dots, x_n, \eta_1, \dots, \eta_m]$ and let $g$ denote the map $\oss \hr E$ making $E$ and $\oss$-superalgebra.  
We will now prove that $g$ has a section $f': E \to \whos$ such that $f' \circ g = 1_{\oss}$. 


Consider the extension 
 \[ 0 \to \frak{m}_s^{n} \to \oss \to \oss/\frak{m}_s^{n} \to 0  \]  
 and note that $\oss/\mfr_s^n$ is a local, Artin superalgebra. We have a map $E \to \oss/\frak{m}_s^{n}$ given by the composition $E \to \blam \to \whos \to \oss/\mfr_s^n$.
Since $E$ is an \'etale $\blam$-superalgebra, the map $E \to \oss/\frak{m}_s^{n}$ has a unique lift $f': E \to \oss$. In particular, $f' \circ g=1_{\oss}$ is a section of $g$. 


Now define
$S'$ to be the fiber product of the diagram 
\begin{center}
\begin{tikzcd}
S' \arrow[r] \arrow[d, hook, "i"] & \spe \oss \arrow[d, hook, "f'"]     \\
\spe \blam \arrow[r]         & \mathbb{A}_{\oss}^{n|m} = \spe E
\end{tikzcd}
\end{center}
and set $\xi = i^* \xil$. 
\newline

For (2): The map  $S' \to \spe \oss$ is \'etale because the map $\spe \blam \to \mathbb{A}_{\oss}^{m|n}$ is \'etale. 
\newline

For (3):  This is obvious since $\xi=i^*(\xil)$ and  $\xil \in \mxx(\spe \blam)$ was such that $\xil \vert_{\spe (\oss/\mfr_s^n)} \cong \whxi \vert_{\spe (\whos/\mfr_s^n)}$ as objects in $\mxx$ over $\spe (\oss/\mfr_s^n) \cong \spe (\whos/\mfr_s^n)$. 
\end{proof}

\section{Artin algebraization} 

In this section, we will prove super Artin algebraization
following the ideas of \cite{conrad2002approximation}, 
except that since we have access already to the bosonic
result, we may replace 
the difficult arguments involving associated gradeds 
by induction on the odd nilpotence degree.



In \cite[Theorem 3.2]{conrad2002approximation}, Conrad and de Jong formulate a more general
version of Artin approximation 
which requires 
only  that 
the initial data is an object over a complete, 
local, Noetherian ring.
They also proved more; indeed
the most subtle and difficult part
of their theorem
establishes an isomorphism
 $\on{gr}_{\mfr}(R) \cong \on{gr}_{\mfr_u}(A)$ of graded algebras,
 which was essential for their proof
 of Artin algebraization.  We will
 not require such a statement
 to prove super Artin algebraization (given 
 the bosonic Artin approximation result). 

 \begin{theorem} \label{CDJ} Let $\mathcal{X}$ be a limit preserving superstack. Let $(R, \frak{m})$ be a complete, local, Noetherian superalgebra, and let $\whxi \in \mxx(\spe R)$. Then for every integer $n \ge 1$, there exists 
 
 \begin{itemize} 
 
 \item[\emph{(1)}] a finite type affine superscheme $\on{Spec} A$ and a closed point  $u \in \on{Spec} A$, 
 
 \item[\emph{(2)}] an object $\xi_A$ of $\mc{X}$ over $\on{Spec} A$, 
 
 \item[\emph{(3)}] an isomorphism $\alpha_n: R/\frak{m}^{n} \cong A/\frak{m}_u^{n}$, and
 
 \item[\emph{(4)}] an isomorphism $\wh{\xi} \vert_{\on{Spec}(R/\frak{m}^{n})} \cong \xi_A \vert_{\on{Spec}(A/\frak{m}_u^{n})}$ of objects in $\mxx$ over $\on{Spec}(R/\frak{m}^{n}) \cong \on{Spec}(A/\frak{m}_u^{n})$. 
 
 
 
 \end{itemize} 
 
 \end{theorem} 
 

 \begin{proof} 
 We follow 
\cite{conrad2002approximation}, drawing
also from the exposition \cite{alper2015artin}.  We recall that any superalgebra is the direct limit of its finitely generated sub-superalgebras. In particular, $R = \varinjlim B_{\lambda}$ where $B_{\lambda}$ are finitely generated (over $k$) sub-superalgebras of $R$.

Since $\mxx$ is limit preserving there exists  large enough $\lambda$, and an object $\xil \in \mxx(\spe \blam)$ such that $\whxi \to \xil$ is an arrow in $\mxx$ over the natural map $\spe R \to \spe \blam$.  
Let $B=\blam$, $V = \spe B$, let $v \in \spe B$ be the image of the unique closed point in $\spe R$. With this notation, let $f$ denote the natural map $\ovv \to R$, and let $\wh{f}$ denote the induced map $\whov \to R$ on the completions. We may add enough
generators to $\blam$ so that the induced map $\wh{f}_*: t_{\ovv}^* \to t_R^*$ on Zariski cotangent spaces is surjective. The map $\wh{f}$ is then surjective by  Lemma \ref{surjcot}


Our next goal is to construct $(\spe A, u)$ and $\xi_A$. Choose a resolution 
 \begin{equation} \label{resolution} (\wh{\alpha}, \wh{\beta}): \ovv^{\oplus r|r'} \overset{\wh{\alpha}}{\longrightarrow} \ovv^{\oplus s|s'} \overset{\wh{\beta}}{\longrightarrow} \ovv \to R  \to 0 \end{equation}
of $R$ and consider the (limit-preserving) functor 
 \begin{align*}
     F & : (\sS/V) \to \on{Sets} \\ 
     {} & (T \to V) \mapsto \{ \text{complexes} \ (\alpha, \beta): \mc{O}_T^{\oplus r|r'} \overset{\alpha}{\longrightarrow} \mc{O}_T^{\oplus s|s'} \overset{\beta}{\longrightarrow} \mc{O}_T \}
 \end{align*}
Observe that \eqref{resolution} is an element of $F(\spe \ovv)$. Since $F$ is limit preserving, we can apply Artin approximation, Theorem \ref{artapp}. This gives for each $n \ge 1$, a superscheme $(V'= \spe B', v')$, an \'etale morphism $(V', v') \to (V,v): v' \to v$, and an object
\[ (\alpha', \beta'): (B')^{\oplus r|r'} \overset{\alpha'}{\longrightarrow} (B')^{\oplus s|s'} \overset{\beta'}{\longrightarrow} B' \] 
in $\xi' \in F(V')$ with the following property: Let $(\alpha_n', \beta_n')$ denote the resolution  modulo $\frk{B'}^{n}$ and let $(\wh{\alpha}_n, \wh{\beta}_n)$ denote the resolution in \eqref{resolution} modulo $\mfr_R^n$ . Then Artin approximation says that $(\wh{\alpha}_n, \wh{\beta}_n) \cong (\alpha_n', \beta_n')$ as objects in $F$. This, moreover, implies that 
\[ R/\frk{R}^{n}= \on{coker}(\wh{\beta}_n) \cong \on{coker}(\beta_n').  \]

Now let $U = \mathbb{V}(\on{im}(\beta')) \subset \spe B'$ be the closed sub-superscheme cut out by $\on{im}(\beta')$ and set $u=v' \in U$. Let $A$ be the superalgebra for which $U=\spe A$. This gives a resolution
\begin{equation} \label{bresol} (B')^{\oplus r|r'} \overset{\alpha'}{\longrightarrow} (B')^{\oplus s|s'} \overset{\beta'}{\longrightarrow} B' \to A \to 0.  \end{equation} 
The composition
\[ \xi_A: U \hr V' \to V \overset{\xi_V}{\longrightarrow} \mxx \]
then defines an object $\xi_A \in \mxx(U)$. Since the resolution in \eqref{bresol} is equivalent to the resolution in \eqref{resolution}, we have that $R/\frk{R}^{n} = \on{coker}(\wh{\beta}_n) \cong \on{coker}(\beta_n') = A/\frk{u}^{n}$ and $\xi_A \vert_{\spe(A/\frk{u}^{n}} \cong \wh{\xi} \vert_{\spe(R/\frk{R}^{n})}$. 

\end{proof}

\begin{theorem}[Artin algebraization] \label{superartalg}  Let $\mxx$ be a limit-preserving superstack and let $(R, \mr)$ be a local, complete, Noetherian superalgebra. Let $\whxi \in \mxx(\spe R)$ be formally versal. Then there exists 
\begin{enumerate}
    \item an affine finite type superscheme $\spe A$ and a closed point $u \in \spe A$, 
    \item an object $\xi_A \in \mxx(\spe A)$, 
    \item an isomorphism $\alpha: R \lgr \wh{A}_{\frk{u}}$ of superalgebras, and
    \item a compatible family of isomorphisms $\wh{\xi} \vert_{\spe(R/\frk{R}^n)} \cong \xi_A \vert_{\spe(A_{\frk{u}}/\frk{u}^n)}$ in $\mxx$ over $R/\mr^{n} \cong A_{\frk{u}}/\frk{u}^{n}$ for all $n \ge 1$. 
\end{enumerate}

\end{theorem}

\begin{proof} Let $r$ be the smallest integer for which $\jar^{r+1}=0$. We can apply Conrad-de Jong approximation at $N=r$ to obtain a finite type pointed superscheme $(\spe A, u)$, an object $\xi_A \in \mxx(\spe A)$, and an isomorphism 
\[ \iota_i: \wh{\xi} \vert_{\spe(R/\frak{m}^i)} \lgr \xi_A \vert_{\spe(A/\frk{u}^i)} \]
over
\[ \alpha_i: \spe A/\frk{u}^i \lgr \spe R/\frk{R}^i \] 
for all $i=1, \dots, r$. 

Since $\wh{\xi}$ is formally versal, we have for each $n \ge 1$ an arrow completing the diagram, 
\begin{center}
\begin{tikzcd}
\spe A/\frk{u}^n  \arrow[d] \arrow[r, "\alpha_n"]                           & \spe R \arrow[d, "\wh{\xi}"] \\
\spe A/\frk{u}^{n+1} \arrow[r, "\xi_A"'] \arrow[ru, "\alpha_{n+1}", dashed] & \mxx  \end{tikzcd}
\end{center}
Taking the projective limit we obtain a morphism $
\wh{\alpha}: R \to \wh{A}_{\frk{u}}$. 

We claim that the morphism $\wh{\alpha}$ is surjective. 
Indeed, the morphism $\wh{\alpha}$ descends to an isomorphism $R/\frak{m}^r \cong \wh{A}_{\frk{u}}/\frk{u}^r$,
hence induces an isomorphism $\mr/\mr^2 \to \muu/\muu^2$ of Zariski cotangent spaces. In particular, the map $\wh{\alpha}$ induces a surjection between Zariski cotangent spaces. The claim that $\wh{\alpha}$ is surjective now follows from  Lemma \ref{surjcot}.  

We now claim that $\wh{\alpha}$ is injective. Let $M= \on{ker}(\wh{\alpha})$. 
By applying the classical Artin algebraization theorem to the bosonic reductions of the objects appearing in the statement of the Theorem, we get an isomorphism \[ R_{\bos} \cong (\wh{A}_{\frk{u}})_{\bos}.  \] Thus $M/\jar=(0)$ and so $M \subset \jar$. Let $m \in M$  and let $r(m)$ largest integer for which \linebreak $m = m \ \  \on{mod} \ \ \mr^{r(m)+1}$. We know such an integer $r(m)$ exists because  $m \in \jar$. Moreover, we must have that $r(m) \le r$. Thus,  $\alpha_{r(m)}(m)=0$ and so $m=0$ since the $\alpha_i$ are isomorphisms.

\end{proof}

\section{Flenner's criterion for openness of formal versality}

\subsection{Extension functor} 
Here we introduce Flenner's 
$\mathrm{Ex}$ functors, following \cite{flenner1981kriterium}
(and somewhat expanding upon the exposition). Throughout this section, let $T=\spe R$ be a finite type affine superscheme.

\begin{definition} For every $v \in \mxx(T)$, and for every coherent $\ot$-module $\mm$, we define $\pmb{\ex}(v, \mm)$ to be the groupoid of pairs $(i: T \hr T', v \to v')$ with $i: T \hr T'$ an extension of $T$ by $\mm$ and $v \to v'$ an arrow in $\mxx$ over $i$. 

\end{definition}

Let
\begin{equation}  \ex(v, -): \on{Coh}(T) \to \on{Set} 
\end{equation} 
be the functor sending a coherent $\ot$-module $\mm$ to the set of isomorphism classes in $\pmb{\ex}(v, \mm)$. 
In general, $\on{Ex}(v, \mm)$ will \emph{not} have a $\got$-module structure. However, there does exist in $\on{Ex}(v, \mm)$ a natural ``zero object", the pair $(T[\mm],  v[\mm])$, where $v[\mm]:= v \times_T T[\mm]$ is the trivial $\mm$-extension of $v$ over $T[\mm]$. An object $(i: T \hr T', v \to v')$ in $\ex(v, \mm)$ is trivial if there exists a section $s: T' \to T$ such that $is=1_T$ and an arrow $v' \to v$ over $s$. 



We have the following natural morphisms of pointed sets (the point is the zero object in each term): 

\begin{equation} \label{morphismseq} 
\begin{tikzcd}
{\derp(\ot, \mm)} \arrow[r, "c"]  & D_v(\mm) \arrow[r, "i"]  & {\ex(v, \mm)} \arrow[r, "\pi"]  & {\ex(T, \mm)}.  \\
\end{tikzcd}
\end{equation}

We remind the reader that $D_v(\mm) = \omxx_v(T[\mm])$ and that $\ex(T, \mm)$ is the set of isomorphism classes of extensions of $T$ by $\mm$.


\begin{itemize}
    \item 
The map $c$ sends a derivation $d$ in $\on{Der}(\ot, \mm)$ to the pullback of $v[\mm]$ under the automorphism $1 + d: T[\mm] \to T[\mm]$, that is 
\[ d \mapsto  c(d)= (1 +d)^*(v[\mm]).  \] 

\item 
The map $i$ takes $v \to v'$ in $D_v(\mm)$ to the pair $(i: T \hr T[\mm], v \to v')$ in $\ex(v, \mm)$. Note that $\on{Der}(\got,\mm)$ acts trivially on $\ex(v,\mm)$ but non-trivially on $D_v(\mm)$. 

\item The map $\pi$ takes the pair $(i: T \hr T', v \to v') \in \ex(v,\mm)$ to $i: T \hr T' $ in $ \ex(T,\mm)$. 
\end{itemize}


\begin{lemma}[Compare to {\cite[Lemma 2.2]{flenner1981kriterium}}] \label{properties}  Using the notation above and assuming that \emph{(S1)} and \emph{(S2)} hold for $\mxx$, we have: 
\begin{enumerate}
\item[\emph{1}.] $\pi(\on{Ex}(v, \mm)) \subset \on{Ex}(T,\mm)$ is an $\got$-submodule
\item[\emph{2}.] $\pi \circ i=0, i \circ c=0$, and
\item[\emph{3}.] $\on{Coker}(c) \to \pi^{-1}(0)$ is bijective. 
\end{enumerate} 
\end{lemma} 

\begin{proof} For (1):
The fact that $\pi(\ex(v,\mm))$ is closed under addition can be seen from the diagram 
\begin{center}
    \begin{tikzcd}
    \ex(v,\mm) \times \ex(v,\mm)  \arrow[d] & \arrow[l,"{p}"] \ex(v, \mm \times \mm) \arrow[r, "{\on{add}}"] \arrow[d] & \ex(v, \mm) \arrow[d] \\ 
    \ex(T, \mm) \times \ex(T,\mm) &  \arrow[l, "{\sim}"] \ex(T, \mm \times \mm) \arrow[r, "{\on{add}}"] & \ex(T,\mm)
    \end{tikzcd}
\end{center}
where $p$ is surjective by (S1)a. 
\newline

For (2):  Let $v' \in D_v(\mm)$, then $i(v')$ is the pair $(i: T \hr T[\mm], v \to v')$ in $\ex(v, \mm)$ and $\pi(i(v'))$ is the object $i: T \hr T[\mm]$ in $\ex(T, \mm)$. Thus, $\pi \circ i=0$.  Let $d \in \on{Der}(\ot, \mm)$, then $c(d)=(1 + d)^*v[\mm]$ over $(1+d)(T[\mm])$. 
The image $i(c(d))$ is the pair \[ (i: T \hr (1+d)\left(T[\mm] \right), v \to (1+d)^*v[\mm]) \]
in $\ex(v, \mm)$. In $\ex(v, \mm)$, we have that $(1+d)\left( T[\mm] \right) \cong T[\mm]$ and $(1+d)^*v[\mm] \cong v[\mm]$. 
Therefore, $i \circ c=0$.
\newline

For (3): Define $D_v(\mm) \to \pi^{-1}(0)$ by $(v \to v') \mapsto (i: T \hr T[\mm], v \to v')$. This map is clearly surjective. The ``kernel" of this map are arrows $v \to (1+ d)^*v[\mm]$ for all $d \in \on{Der}(\ot,\mm)$.  Therefore, we get a bijection $\on{coker}(c) \to \pi^{-1}(0)$. 
\end{proof}

Coherence properties of $\derp(\ot,\mm)$ and
$\ex(T, \mm)$ are ensured by the following
relationship with the cotangent complex 
$L_{T}^{\bullet}$  (see e.g. \cite[Definition 3.6]{flenner1992analytic}. 
For any $\ot$-module $\mm$, we have the \emph{cotangent functors}
\[  T_{T}^i(\mm):= \on{Ext}_T^i(L_{T}^{\bullet}, \mm). \]

\begin{theorem}[{\cite[Theorem 4.6]{flenner1992analytic}}] There are canonical isomorphisms 
\[ \ex(T, \mm)  \to T_{T}^1(\mm)_{ev}, \ \ \ \on{Der}(\ot, \mm) \to T_{T}^0(\mm)_{ev}\]
where $T_{T}^i(\mm)_{ev}$ denotes the set of even elements in $T_{T}^i(\mm)$ for $i=0,1$ 
\end{theorem}


 We consider the following sheaves: 

\begin{enumerate} 

\item Let $\mc{E}x(v, \mm)$ be the sheaf of pointed sets over $T$ associated to the presheaf  \linebreak $T \supset U \mapsto \on{Ex}(v \vert_U, \mc{M} \vert_U)$.

    \item Let $\mc{D}_v(\mm)$ be the sheaf over $T$ associated to the module $D_v(\mm)$.
    
 \item Let $\mc{D}er(\ot, \mm)$ be the sheaf of over $T$ associated to the module $\on{Der}(\ot, \mm)$. 
    
\end{enumerate} 

$\exc(v, \mm)_t$, etc. will denote the stalk at a closed point $t \in T$.  Note that Lemma \ref{properties} holds for the sheaves $\mc{D}er(\ot, \mm)$, $\exc(v, \mm)$ and $\mc{D}_v(\mm)$.

The main result we need about the Ex functors is 
Flenner's criterion for formal versality: 

\begin{theorem}[Compare {\cite[Theorem 3.2]{flenner1981kriterium}}] \label{flennerm} Let $v$ be an object in $\mxx$ over an affine superscheme $T$. Assume that \emph{(S1)} and \emph{(S2)} are satisfied. Then the following statements are equivalent. 
\begin{enumerate}
    \item[\emph{(1)}] $v$ is formally versal, 
    \item[\emph{(2)}] $\exc(v, k(t))_t=\exc(v, \Pi k(t))_t=0$. 
    \item[\emph{(3)}] $\exc(v, \mm)_t=0$ for every coherent $\otr$ module $\mm$.

\end{enumerate}
\end{theorem}

We prove this result in Appendix \ref{flenner proof
appendix}.  The proof follows Flenner's original argument
except at the (important) lemma \ref{copr}, 
where we are able to reduce
to the bosonic case rather than recapitulating the proof.



\subsection{Criteria for openness of formal versality} 

We turn to the proof of
openness of formal versality. Given a topological space $X$, we say that a subset $E \subseteq X$ is is \emph{ind-constructible} in $X$ if, for all $x \in X$, there exists an open neighborhood $U \ni x$, such that $E \cap U$ is a union of constructible sets. A constructible set is a finite disjoint union of locally closed sets. 
We will need the following constructibility criterion (see e.g., \cite[Proposition 1.9.10]{grothendieck1965elements} for the ordinary version). 

\begin{proposition} \label{indcon} Let $X$ be a locally Noetherian superscheme and let $E$ be a subset of $X$. For $E$ to be ind-constructible it is necessary and sufficient that for all  $x \in X$, the intersection $E \cap \ov{ \{ x \} }$ is a neighborhood of $x$ contained in $\ov{ \{x \} }$. 
\end{proposition}


\begin{lemma}[Compare to {\cite[Lemma 4.1]{flenner1981kriterium}}] \label{tech} For every $T \in \sS$, and $v \in \mxx(T)$, suppose the following conditions are satisfied
\begin{enumerate}
    \item[\emph{1.}] Let $T'$ be a closed subspace of $T$. Suppose also that $T'$ is integral. Then there exists a dense, open Zariski subset $V' \subseteq T'$ such that 
    \[ \on{Supp}(\exc(v, \mc{O}_{T'})) \cap V' = \{ t \in V' : \exc(v, k(t))_t \neq 0 \}  \]
   \[ \on{Supp}(\exc(v, \Pi \mc{O}_{T'})) \cap V' = \{ t \in V' : \exc(v, \Pi k(t))_t \neq 0 \}  \]
  In other words, $V'$ as the special property that $\exc(v, \mm)_t=0$ (with $t \in V'$)  only when  $\mm$ is $\mc{O}_{T'}$ or when $\mm$ is $\Pi \mc{O}_{T'}$.

    \item[\emph{2.}] For every coherent $\otr$-module $\mm$, the support of $\exc(v,\mm)$ is Zariski closed in $T$. 
    
\end{enumerate}

Then the set of points $t \in T$ at which $v$ is formally versal over $\mxx$ is a Zariski open subset of $T$. 
\end{lemma}

\begin{proof} 

Let
\[V= \{ t \in T \ | \ \exc(v, k(t))_t=\exc(v, \Pi k)_t= 0 \}.  \]
From Theorem \ref{flennerm}, we have that $V$ is the set of closed points in $T$ at which $v$ is formally versal and that \[ V= \bigcap_{\mm \in \on{Coh}(T)} \{ t \in T \ | \ \exc(v, \mm)_t =0 \}.  \]

The goal is to prove that $V$ is open.  We will do this by applying the following theorem from EGA, \cite[Theorem 1.10.1]{grothendieck1965elements}: If a set is ind-constructible and stable under generalization, then it is open. 

We first prove that $V$ is stable under generalization. Let $t \in V$, and suppose that $t \in \ov{\{ t' \} }$ for some $t' \in T$. We claim that $t' \in V$. Fix some $\mm$. Since $\exc(v, \mm)$ is closed (by assumption (2)), 
the set  
\[ V_M : = T - \on{Supp}(\exc(v, \mm)) = \{ t \in T \ | \  \exc(v, \mm)_t =0 \} \]
is open. Since $V_M$ is open, it is stable under generalization, and so $t' \in V_M$.
Repeating this argument for every $\mm$, we find that $t' \in V$.

We now prove that $V$ is ind-constructible using Proposition \ref{indcon}.  Let $T'$ be an integral closed subscheme of $T$ which meets $V$.  Let
\begin{align*}
V'' & = T' - \left ( \on{Supp}(\exc(v, \mc{O}_{T'})) \cap \on{Supp}(\exc(v, \Pi \mc{O}_{T'}))  \right ) \\
{} & =\{ t \in T' \ | \ \exc(v, \mc{O}_{T'})_{t}=\exc(v, \Pi \mc{O}_{T'})_{t} = 0 \}.
\end{align*}   
$V''$ is open by condition (2). $V'' \neq \emptyset$ since $V \cap T' \neq \emptyset$ and so by Proposition \ref{flennerm}  there exists at least one $t \in T'$ such that $\ex(v, \mm)_t=0$ for all  $\mm$. Set
\[ U' = V'' \cap V'. \] 
It follows from the definition of $V'$, that the set $U'$ is equal to the subset \linebreak $\{ t \in V' \ | \ \exc(v,k(t))=\exc(v, \Pi k(t))=0 \} $ of $V \cap T$. Since $V'$ is Zariski-dense in $T'$, so is $U'$, and, thus, $V$ is ind-constructible.

\end{proof}

For $v$ an object in $\mxx$ over $T$, and $\mm$ a coherent $\otr$-module, define  $\mc{O}b_v(\mm)$ to be the sheaf over $T$ associated to the presheaf sending $U \subseteq T$ to $\ex(U, \mm \vert_U)/\pi(\ex(v \vert_U, \mm \vert_U)$.  $\mc{O}b_v(-)$ is clearly functorial in $\mm$. 

We now introduce the following conditions. \footnote{The labels (A3) and (A4) are chosen to match Theorem \ref{mainth}. }
\newline

\textbf{Condition} (A3). Coherence of $\mc{O}b$: For all $v, T, \mc{M}$ as above, $\mc{O}b_v(\mc{M})$ is a coherent $\ot$-module, and
\newline

\textbf{Condition} (A4). Constructibility: Let $v, T$ be as above and suppose that $T' \subseteq T$ is an irreducible reduced subspace of $T$. Then there is a Zariski open dense subset $V'\subseteq T'$ such that for each $t \in V'$, the canonical morphisms 
\begin{itemize}
    \item[(a)] $\mc{D}_v(\otpr) \otimes k(t) \to \mc{D}_v(\otpr \otimes k(t))_t$  and $\mc{D}_v(\Pi \otpr) \otimes k(t) \to \mc{D}_v(\Pi \otpr \otimes k(t))_t$ are  bijective, and
    \item[(b)] $\mc{O}b_v(\otpr) \otimes k(t) \to \mc{O}b_v(\otpr \otimes k(t))_t$ and $\mc{O}b_v(\Pi \otpr) \otimes k(t) \to \mc{O}b_v(\Pi \otpr \otimes k(t))_t$  are injective. 
\end{itemize}



\begin{theorem}[Compare {\cite[Theorem 4.3]{flenner1981kriterium}}]  \label{flenneropen}  Let $\mxx$ be a superstack satisfying conditions \emph{(S1), (S2), (A3)} and \emph{(A4)}. 
Then for every superscheme $T$ and every $v \in \mxx(T)$, the subset of points $t$ in $T$ at which $v$ is formally versal form a Zariski-open subset of $T$. 
\end{theorem}

\begin{proof}

The goal is to prove that the two conditions of Lemma \ref{tech} are satisfied.

Let $T'$ and $V'$ be as in condition (A2). 
We first prove that the subset $V' \subset T'$ is equal to the subset $V'$ defined in Lemma \ref{tech}. 
We begin by proving that $V'$ has the following properties 
\begin{itemize}

\item[(i).] $\mc{O}b_v(\mc{O}_{T'}) \vert_{V'}$ and $\obv( \Pi \mc{O}_{T'}) \vert_{V'}$  are locally free, and 

\item[(ii). ] the natural maps
\begin{equation} \label{dermap}
   \cder(\ot, \otpr) \otimes k(t)  \to \cder(\ot, k(t))_t,  \end{equation}
   and
   \begin{equation} \label{omap}
    \exc(T, \otpr) \otimes k(t)  \to \exc(T, k(t))_t
\end{equation}
are bijective for all $t \in V'$, and analogously for $\cder(\ot, \Pi \otpr)$ and $\exc(T, \Pi \otpr)$. 
\end{itemize} 

Henceforth, all statements we make in terms of $\otpr$ hold for $\Pi \otpr$ using the exact same proof.

Since $\obv(\otpr)$ is a coherent $\otpr$-module by (A1), we can pass to a dense open subset $W \subset V'$, on which $\obv(\otpr) \vert_W$ is locally free. Clearly, condition (A2) holds for $V' \cap W$ and so we can just redefine $V'= V' \cap W$. This proves the first claim

For the second claim, identify $\cder(\ot, -)$ and $\exc(\ot, -)$ with $\mc{E}xt_{\ot}^0(L^{\bullet}, -)_{ev}$ and  $\mc{E}xt_{\ot}^1(L^{\bullet}, -)_{ev}$, respectively.  Under these identifications, the maps \eqref{dermap} and \eqref{omap} are well-known to be bijections. This proves the second claim.

We now consider the diagram 
\begin{equation} \label{digofv}
\begin{tikzcd}
{\cder(\ot, \otpr)} \arrow[d] \arrow[r, "c"] & \mc{D}_v(\otpr) \arrow[r] \arrow[d] & {\exc(v, \otpr)} \arrow[r, "\pi"] \arrow[d] & {\exc(T, \otpr)} \arrow[r] \arrow[d] & \mc{O}b_v(\otpr) \arrow[d] \\
{\cder(\ot, k(t))} \arrow[r, "c_t"]             & \mc{D}_v(k(t)) \arrow[r]               & {\exc(v, k(t))} \arrow[r, "\pi_t"]             & {\exc(T, k(t))} \arrow[r]               & \mc{O}b_v(k(t))              
\end{tikzcd}
\end{equation}
where all terms are assumed to be restricted to $V'$. 
We will now prove that the canonical morphisms 
\begin{align*}
    \on{coker}(c) \otimes k(t) & \to \on{coker}(c_t)_t \\ 
    \on{Im}(\pi) \otimes k(t) & \to \on{Im}(\pi_t)_t
\end{align*}
are bijections for all $t \in V'$. 

All terms in the top row of diagram \eqref{digofv} (other than $\exc(v, \otpr)$) are coherent modules over the Noetherian domain $V'$ and so we can apply generic flatness to all the objects in the top row, as well as to all their images and kernels and get an open subset $W \subset V'$ on which  the functor $- \otimes_{\mc{O}_{V'}} k(t)$  is exact. Redefine $V'= V' \cap W$. 

For the remainder of the proof, we set $A= \cder(\ot, \otpr)$, $B= \mc{D}_v(\otpr)$, $C= \exc(v, \otpr)$, $D=\exc(T, \otpr)$ and $E= \mc{O}b_v(\otpr)$. We also set
$A_k= \derp(\ot, k(t))$, $B_k = \mc{D}_v(k(t))$, etc..

We begin by proving that $\on{coker}(c) \otimes k(t) \to \on{coker}(c_t)_t$ is bijective for all $t \in V'$. Consider the short exact sequence 
\begin{equation} \label{sflt} 0 \to A \overset{c}{\to} B \to B/\on{im}(c) \to 0.  \end{equation}  
Since $A, B$  are flat, we can tensor \eqref{sflt} with $k(t)$, and get a short exact sequence
\begin{equation} \label{tybr}
    0 \to A \otimes k(t) \overset{c \otimes 1}{\longrightarrow} B \otimes k(t) \to \left( B/\on{im}(c) \right) \otimes k(t) \to 0.  
\end{equation}
From which it follows that $ \on{coker}(c) \otimes k(t)= \on{coker}(c \otimes 1)$. We have that (1) $B \otimes k(t) = (B_k)_t$ by condition (A2)a, and (2) that $A \otimes k(t)= (A_k)_t$  by bijectivity of the map \eqref{dermap}. From these identifications, we get the following diagram 
\begin{equation} \label{biglocal}
    \begin{tikzcd}
    0 \arrow[r] & A \otimes k(t) \arrow[d, no head, double] \arrow[r, "c \otimes 1"] & B \otimes k(t) \arrow[d, no head, double] \arrow[r] & \on{coker}(c \otimes 1)  \arrow[d, no head, double] \arrow[r] & 0 \\
    0 \arrow[r] & (A_k)_t \arrow[r] & (B_k)_t \arrow[r] & (B_k)_t/(A_k)_t \arrow[r] & 0. 
    \end{tikzcd}
\end{equation} 
Here $(B_k/A_k)_t = (B_k)_t/(A_k)_t$ by flatness of localization, and so $(B_k)_t/(A_k)_t= \on{coker}(c_t)_t$ since $B_k/A_k=\on{coker}(c_t)$.  This proves that $\on{coker}(c) \otimes k(t) = \on{coker}(c_t)_t$.

We will now prove that the canonical morphism $\on{im}(\pi) \otimes k(t) \to \on{im}(\pi_t)_t$ is a bijection. 

Consider the short exact sequence  
\begin{equation} \label{imageseq}
    0 \to \on{im}(\pi) \to D \to D/\on{im}(\pi) \to 0
\end{equation}
where $E=D/\on{im}(\pi)$ by definition.  The term $\on{im}(\pi)$ is a $\mc{O}_{V'}$-module by Lemma \ref{properties}. All terms in \eqref{imageseq} are flat (by generic flatness) and coherent (by condition (A1)). Since all terms in  \eqref{imageseq} are flat, we can tensor with $k(t)$ and get a short exact sequence 
\begin{equation}
    0 \to \on{im}(\pi) \otimes k(t) \to D \otimes k(t) \to E \otimes k(t) \to 0 
\end{equation}
from which it follows that $E \otimes k(t) = \left(D \otimes k(t) \right)/ \left( \on{im}(\pi) \otimes k(t) \right)$.  We have that (1) $D \otimes k(t) = (D_k)_t$ by bijectivity of \eqref{omap}, and (2) that $E \otimes k(t) \subset (E_k)_t$ by (A2)b. 

Using the above identifications, we get the following diagram
\begin{equation} \label{ishlocal}
    \begin{tikzcd}
    0 \arrow[r] & \on{im}(\pi) \ \otimes k(t) \arrow[d] \arrow[r] & D \otimes k(t) \arrow[d, no head, double] \arrow[r] & E \otimes k(t)  \arrow[d, hook] \arrow[r] & 0 \\
    0 \arrow[r] & (\on{im}(\pi_t))_t \arrow[r] & (D_k)_t \arrow[r] & (E_k)_t \arrow[r] & 0. 
    \end{tikzcd}
\end{equation} 
The map $\imp \otimes k(t) \to \on{im}(\pi_t)_t$ is bijective by the 5 lemma. 


  
We will now prove that $V'$ is the from Lemma \ref{tech}. Recall that the subset $V'$ in Lemma \ref{tech} is defined by the following property: For all $t \in V'$, $C=0$ implies that $(C_k)_t=0$. From  \eqref{digofv} we get the following diagram of pointed sets: 
\begin{equation}
\begin{tikzcd}
    0 \arrow[r] & \on{coker}(c) \arrow[r, hook] \arrow[d] & C \arrow[r] \arrow[d] & \on{im}(\pi) \arrow[r] \arrow[d] & 0 \\
    0 \arrow[r] & (\on{coker}(c_t)_t \arrow[r, hook] & (C_k)_t \arrow[r] & \on{im}(\pi_t)_t \arrow[r] & 0
    \end{tikzcd}
\end{equation}

If $C =0$, then $\on{coker}(c)=0$ and $\on{im}(\pi)= 0$. Tensoring $\on{coker}(c)$ and $\on{im}(\pi)$ with $k(t)$, we find that $0= \on{coker}(c) \otimes k(t) = \on{coker}(c_t)_t$ and that $0 =\imp \otimes k(t)=\on{im}(\pi_t)_t$. Thus, $(C_k)_t=0$. This proves that $V'$ is a subset of $T'$ as in Lemma \ref{tech}.


We will now prove that condition two of Lemma \ref{tech} is satisfied: For every coherent $\otr$-module $\mm$ the support of $\ex(v, \mm)$ is a closed subset of $T$. For each such $\mm$, consider the sequence of pointed sets:
\begin{equation}
    0 \to \cder(\ot, \mm) \overset{c}{\to} \mc{D}_v(\mm) \to \exc(v, \mm) \overset{\pi}{\to} \exc(T, \mm) \to \mc{O}b_v(\mm).
\end{equation}
We then have a sequence of pointed sets, 
\begin{equation}
    0 \to \on{coker}(c) \to \exc(v, \mm) \to \imp \to 0
\end{equation}
and so we have that 
\[ \on{Supp}(\exc(v, \mm))= \on{Supp}(\on{coker}(c)) \cup \on{Supp}(\imp). \] 
Here $\on{coker}(c)$ is a coherent $\ot$-module by (S1)b and $\imp$ is coherent because it is a submodule of $\exc(T, \mm)$. Therefore,  $\on{Supp}(\exc(v, \mm))$ since it is a union of the supports of coherent modules.  \end{proof}

\section{Formal smoothness  from formal versality} 

Let $C^*$ denote the category of local, Noetherian, Henselian superalgebras with residue field at the unique closed point all equal to $k$. Throughout this section, let $R$ be a finite type superalgebra in $C^*$.


\begin{definition}[Versal] An object $v$ in $\mxx$ over $\spe R$ is called versal if for every surjection $A' \to A$ in $C^*$ there exists an arrow completing the diagram 
\begin{equation} \label{versalitycond}
\begin{tikzcd} 
 \spe A \arrow[d] \arrow[r]           & \spe R \arrow[d, "v"] \\
 \spe A' \arrow[r] \arrow[ru, dashed] & \mxx                 
\end{tikzcd}
\end{equation}

\end{definition}
The difference between
the lifting condition in \eqref{versalitycond} and formal versality is that versality is tested against
all Henselian superalgebras, rather than just Artin superalgebras.  


\begin{theorem}[Compare {\cite[Theorem 3.3]{artin1974versal}}] \label{versality} 
Let $v$ be a formally versal object in $\mxx$ over $\spe R$. Suppose the following conditions hold over $C^*$. 
\begin{itemize}
    \item[\emph{(i)}] \emph{(S1)} and \emph{(S2)} hold. 
    
    \item[\emph{(ii)}] ($D$ is compatible with completion): Let $A' \to A \to A_0, \  \on{ker}(A' \to A)=M$ be as in \eqref{infextr} with $A_0$ of finite type. Let $a_0 \in \mxx(\spe A_0)$, then \begin{equation} \label{compcomp} D_{a_0}(M) \otimes \wh{A}_0 \lgr \varprojlim D_{a_0}(M/\mfr^n M) \end{equation} 
    
    \item[\emph{(iii)}] Let $A$ be a finite type (ordinary) algebra with an ideal $I$ and let $\ov{A}$ be its $I$-adic completion. Let $a,b \in \mxx(\spe \ov{A})$. If there exists a compatible sequence of isomorphisms $a_n \lgr b_n$ between the truncations in $\mxx(\spe A/I^{n+1})$, then there is an isomorphism $a \lgr b$ compatible with the given isomorphism $a_0 \lgr b_0$. 
\end{itemize}
If $(i), (ii)$ hold, then the lifting condition \eqref{versalitycond} holds for all infinitesimal extensions $A' \to A$. If $(i)-(iii)$ hold, then $v$ is versal. 
\end{theorem}

\begin{proof} 
We follow the first part of Artin's proof of 
\cite[Theorem 3.3]{artin1974versal}), 
concerning infinitesimal extensions, 
but afterwards avoid some difficult
limit arguments by then reducing 
to the bosonic case.


Consider the lifting problem for an infinitesimal extension $A' \to A \to A_0, \on{ker}(A' \to A)=M$. Since $\mxx$ is limit preserving, we may assume that $A'$ is finite type. We may also assume that $M$ is a finitely generated $A_0$-module with $M^2=0$. 
By adding an appropriate number of generators to $R$, we may further assume that the map $R \to A$ is surjective so that $A$ is a finite type $R$-module.  

The goal is to prove that for every map $f: R \to A$ there exists an arrow completing the diagram
\begin{equation} \label{endgame}
\begin{tikzcd}
\spe A \arrow[r] \arrow[d, hook]            & \spe R \arrow[d, "v"] \\
\spe A' \arrow[r] \arrow[ru, dashed] & \mxx                 
\end{tikzcd}
\end{equation}


\begin{lemma} \label{versalitylemma} There exist an arrow completing the diagram
\begin{equation} \label{versalitythm}
\begin{tikzcd}
R \arrow[rr, dashed] \arrow[rrd, swap, "f"] &  & A' \arrow[d] \\
                                 &  & A           
\end{tikzcd}
\end{equation}  

\end{lemma}

Set $A_n'= A'/\mfr_{A'}^n$, $A_n = A/\mfr_A^n, \ \on{ker}(A_n' \to A_n)=M(n)$. Since $v$ is formally versal, there exists for each $n$ an arrow $R \to A_n'$ completing the diagram 
\begin{equation} \label{dig} 
\begin{tikzcd}
R \arrow[rr, dashed] \arrow[rrd, swap, "f_n"] & &  A_n' \arrow[d] \\ 
& & A_n. 
\end{tikzcd}
\end{equation} 
The set of arrows completing \eqref{dig} can be (non-canonically) identified with the set \linebreak $\on{Der}(R, M(n))$. To prove there exists an arrow completing diagram \eqref{versalitythm} it suffices to prove that the inverse system $\{ \on{Der}(R, M(n), \phi_{m,n}: \on{Der}(R, M(m)) \to \on{Der}(R,M(n)) \}$ for $m \ge n$ satisfies the Mittag-Leffler (ML) condition. \footnote{ For each $n$, there exists some $n_0 \ge n$, such that \[ \phi_{n',n}\left(\on{Der}(R, M(n')) \right) = \phi_{n'', n} \left( \on{Der}(R, M(n)) \right) \subset \on{Der}(R, M(n)) \]  for all $n', n'' \ge n_0$. Since $\on{Der}(R, M(n))$ can be identified with the set of arrows completing \eqref{dig}, the \emph{stability} of the system $\phi_{m,n}$ implies that the set of arrows completing \eqref{dig} is also stable. In particular, the set of arrows does not go to the empty set in the limit since it is non-empty for each $n$.  }
To prove that the inverse system $\{ \on{Der}(R, M(n),  \phi_{m,n} \}$ satisfies the ML condition, we will prove that $\on{Der}(R, M(n)$ is a finite length $A_0$-module and that the maps $\phi_{m,n}$ are $A_0$-linear since such a system automatically satisfies the ML condition (see e.g, \cite[Example 9.1.2]{hartshorne2013algebraic}).

We begin by proving that each $M(n)$ is finite length over $A_0$. We find that $M(n)=M/\mfr_{A'}^n \cap M$ by applying the $9$-lemma to the diagram 
\begin{center}
\begin{tikzcd}
            & 0 \arrow[d]                          & 0 \arrow[d]                   & 0 \arrow[d]                &   \\
0 \arrow[r] & \mfr_{A'}^n \cap M \arrow[r] \arrow[d] & \mfr_{A'}^n \arrow[r] \arrow[d] & \mfr_A^n \arrow[r] \arrow[d] & 0 \\
0 \arrow[r] & M \arrow[r] \arrow[d]                & A' \arrow[r] \arrow[d]        & A \arrow[r] \arrow[d]      & 0 \\
0 \arrow[r] & M(n) \arrow[r] \arrow[d]             & A_n' \arrow[r] \arrow[d]      & A_n \arrow[r] \arrow[d]    & 0 \\
            & 0                                    & 0                             & 0                          &  
\end{tikzcd}
\end{center}
In particular, $M(n)$ is naturally a finite $A_0$-module since $M$ is finite over $A_0$.  
To prove that $M(n)$ has finite length over $A_0$, we will construct for it a finite length composition series. Consider
\begin{equation} \label{composition} 0 = \mfr_{A'}^n M(n) \subset \mfr_{A'}^{n-1} M(n) \subset \mfr_{A'}^{m-2} M(n) \subset \cdots \subset \mfr_{A'} M(n) \subset M(n).    \end{equation} 
The quotients $\mfr_{A'}^i M(n) /\mfr_{A'}^{i+1} M(n)$ are finite dimensional super vector spaces over $k=A'/\mfr_{A'}$ and so we can take refinement of \eqref{composition} and get a finite length composition series for $M(n)$. This proves that $M(n)$ has finite length over $A_0$. 

We will now prove that $\on{Der}(R, M(n))$ is a finite length $A_0$-module. The $A_0$-module structure on $\on{Der}(R,M(n))$ is as follows: 
For each $a \in A_0$, and $d \in \on{Der}(R, M(n))$, define $(a \cdot d)(r)=a d(r)$. It is easy to check that $(a \cdot d)(r)$ is a derivation.
To see that $\on{Der}(R, M(n))$ is of finite length over $A_0$, just replace $M(n)$ in \eqref{composition} with $\on{Der}(R, M(n))$.

To prove that the maps $\phi_{m,n} : \on{Der}(R, M(m)) \to \on{Der}(R, M(n))$ are $A_0$-linear, note that the composition $f_0: R \overset{f}{\to} A \to A_0$ gives $M(n)$ an  $R$-module structure, $r \cdot m = f_0(r) \cdot m$.  In particular, $\on{Der}(R, M(n)) \cong \on{Hom}_R(\Omega_R, M(n))$ and
\[ \phi_{m,n}: \on{Hom}_R(\Omega_R, M(m)) \to \on{Hom}_R(\Omega_R, M(n)): (R \to M(m)) \mapsto (R \to M(m) \overset{\ovmn}{\longrightarrow} M(n) \]
where $\ovmn$ is the quotient map 
\[\ovmn: M(m) = M/ \mfr_{A'}^m \cap M \to M(n) = M /\mfr_{A'}^n \cap M = M(m)/\mfr_{A'}^n \cap M(m).\]
The map $\phi_{m,n}$ 
is $A_0$-linear, because $\ovmn$ is $A_0$-linear. 
This proves that $\{ \on{Der}(R, M(n), \phi_{m,n} \}$ is an inverse system of finite length $A_0$-modules and thus satisfies the Mittag-Leffler condition. 
In particular, we get a map $R \to \wh{A}'$ compatible with $f: R \to A$ which can be approximated by a map $\phi: R \to A'$ by Theorem \ref{artapp}.



We will now describe how the lifting problem for $R$ can be reduced to solving the same lifting problem when $A'=A[M]$. Let $b' = \phi^*(v) \in \mxx_a(\spe A')$ and note that $b'$ need not be isomorphic to $a' \in \mxx_a( \spe A')$. The lifting problem for $R$ is thus equivalent to proving there exists $d \in \derm$ such that \begin{equation} \label{goalss} (\phi + d)^*(v) \cong a' \end{equation} in $\mxx_a(\spe A')$. By (S1), the group $D_{a_0}(M)$ acts transitively on the set $\omxx_a(\spe A')$ (see Lemma \ref{modulestructure}) and so there exists $e \in D_{a_0}(M)$ such that $a' \cong b' + e$. 
In particular, to prove \eqref{goalss} it suffices to show there exists $d \in \derm$ such that $d^*(v) \cong e$. To prove such a $d$ exists it suffices to prove that the map $c: \derm \to D_{a_0}(M)$ in \eqref{morphismseq} is surjective.  Since $D_{a_0}(M)=D_a(M)$ by (S1)b and since $\derm = \on{Hom}_{/f}(R, M)$, proving that the map $c$ is surjective is equivalent to solving the lifting problem for $R$ when $A'=A[M]$.


If $M$ has length $(1|0)$ or $(0|1)$, then the map $A_0 \times_k k[M] \to A_0[M]: (a, (s(a), m)) \mapsto (a, m)$, where $s: A_0 \to A_0/\mfr_{A_0} =k$,  is an isomorphism of superalgebras. By (S1)b, this implies that $\omxx_a(\spe A[M])=\omxx_{a_0}(\spe A_0[M])=\omxx(\spe k[M])$. 

Since $k[M]$ is Artin, the lifting condition for $R$ holds by formal versality of $v$. By induction, the lifting property will hold whenever $M$ has finite length. Let $M_n := M/\mfr_{A}^{n+1} M$, then we have just proved that $c_n: \on{Der}(R, M_n) \to D_{a_0}(M_n)$ is surjective for each $n$. We are left to prove that it remains surjective in the limit. 
Both $M_n$ and $R$ are finite length $R$-modules and so $\on{Der}(R, M_n)$ and $D_{a_0}(R)$ are finite-length $R$-modules. In particular the ML condition is satisfied and, 
\[ \wh{c}: \varprojlim_n \on{Der}(R, M_n) \to \varprojlim_n D_{a_0}(M_n) \]  
is surjective. 
Thus, 
\[ 
      \varprojlim \on{Der}(R, M_n) = \varprojlim \on{Hom}_R(\Omega_R, M_n)= \on{Hom}_R(\Omega, M) \otimes \wh{R}  
     \] 
and
    \[ \varprojlim D_{a_0}(M_n)  = D_{a_0}(M) \otimes_{A_0} \wh{\ared} = D_{a_0}(M) \otimes_R \wh{R} 
\]
 by assumptions (ii) of the theorem. Thus $\on{Der}(R, M(n)) \to D_{a_0}(M_n) $ is surjective since its tensor product with $\wh{R}$ is surjective. 
 


This completes the proof that 
the desired lifting holds for
infinitesimal extensions.  We now
use this fact, together with 
Artin approximation for bosonic schemes, to deduce our result.

Let $A' \to A$ be a surjection in $C^*$ and let $f: \spe A \to \spe R$.  By \cite[Theorem 3.3]{artin1974versal},
we can find a lift on the 
bosonic parts.


\begin{center}
\begin{tikzcd}
\spe A_{\bos} \arrow[d] \arrow[r, 
"f_{\bos}"]   & \spe R \arrow[d] \\
\spe A'_{\bos} \arrow[r, "\xi_{A'_{\bos}}"] \arrow[ru, dashed, "g_{\bos}"] & \mxx    \end{tikzcd} 
\end{center}

Since 
$A \to A_{\bos}$ is an infinitesimal extension by  $\jj_{A'}$, 
our arguments above imply 
the existence of $g$ completing the following diagram:
\begin{center}
\begin{tikzcd}
\spe A_{\bos}' \arrow[d] 
\arrow[r, "g_{\bos}"]   
& \spe R \arrow[d] \\
\spe A' \arrow[r, "\xi_{A'}"] \arrow[ru, dashed, "g"] & \mxx    \end{tikzcd} 
\end{center}


However, such $g$ has no 
reason to satisfy 
$g|_{\spe A} = f$, i.e. 
to 
make the upper triangle
of the following diagram
commute 
(although the lower triangle would): 
\begin{equation} \label{entire}
    \begin{tikzcd}
    \spe A \arrow[r, "f"] \arrow[d] & \spe R \arrow[d] \\
    \spe A' \arrow[r, "\xi_A"] \arrow[ru, dashed, "g"] & \mxx.
    \end{tikzcd}
\end{equation}

Instead, note we have
a map
$$g_{\bos} \sqcup f: 
\spe A'_{\bos} \sqcup_{\spe A_{\bos}} \spe A \to \spe R
$$ 
The source is equivalently
$\spe (A'_{\bos} \times_{A_{\bos}} A)$.  Observe that 
$A' \to A'_{\bos} \times_{A_{\bos}} A$
is surjective, since 
$A'_{\bos} \times_{A_{\bos}} A$ 
is generated over $A'_{\bos}$
by the odd part of $A$. 
The kernel is evidently nilpotent.
Thus by the above discussion
we may lift: 

\begin{equation} \label{entire}
    \begin{tikzcd}
    \spe (A'_{\bos} \times_{A_{\bos}} A) 
    \arrow[r, "g_{\bos} \sqcup f"] 
    \arrow[d] & \spe R \arrow[d] \\
    \spe A' \arrow[r, "\xi_A"] \arrow[ru, dashed, "g"] & \mxx.
    \end{tikzcd}
\end{equation}

This $g$ will satisfy 
$g|_{\spe A} = f$, as desired.



\end{proof}

\begin{remark} \label{liftingremark} \normalfont 


Let $S$ be a finite type superscheme and let $s \in S$ be a closed point. 
We recall that the henselization of $\oss$ at $s$ is defined to be $\henss= \varinjlim B_i$ where the $B_i$ are a directed system of \'etale $\mc{O}_{S,s}$-superalgebras. The fact that $\mxx$ is limit preserving, means that
\[ \varinjlim_i \mxx(\spe B_i)) = \mxx(\spe \henss). \] 
Let $v \in \mxx(\mc{O}_{T,t}^h)$ and $a \in \mxx(\oss^h)$, let $f: \mc{O}_{T,t}^h \to \oss^h $ be a morphism in $C^*$, and let $a \to v$ be an arrow in $\mxx$ over $f$.  
 Since $\mxx$ is limit preserving, there exists $i$ large enough such that there exists a factorization $\mc{O}_{T,t}^h \to B_i \to \oss^h$ and an object $a_i \in \mxx(B_i)$ with an arrow $a \to a_i \to v$  over $\mc{O}_{T,t}^h \to B_i \to \oss^h $. 


\end{remark}

Following Artin \cite{artin1974versal}, 
we deduce formal smoothness
from formal versality 
and the following additional conditions. 
\newline


\textbf{Condition} (A5). The sheaf $\mc{D}$ is compatible with completion: If $T_0$ is a reduced finite type scheme, $a_0 \in \mxx(T_0)$, $t \in T_0$ and $\mm$ is a coherent $\oto$-module, then \begin{equation} \label{compcomp} \mc{D}_{a_0}(\mm) \otimes_{\mc{O}_{T_0}} \wh{\mc{O}}_{T_0, t} \lgr \varprojlim \mc{D}_{a_0}(\mm/\mfr_t^n \mm)_t \end{equation} 
\newline 
 
\textbf{Condition} (A6). The sheaf $\mc{D}$ is compatible with \'etale localization: If $e: S_0 \to T_0$ is \'etale, where $e^*(a_0)=b_0$, and $\mm$ a coherent $\oto$-module, then
\begin{equation} \label{compcompl}
    \mc{D}_{b_0}(\mm \otimes_{\oto} \mc{O}_{S_0}) \cong \mc{D}_{a_0}(\mm) \otimes_{\oto} \mc{O}_{S_0}
\end{equation}
\newline

\begin{theorem} \label{formsmooth} Suppose that \emph{(S1)}, \emph{(S2)}, \emph{(A5)} and \emph{(A6)} hold for $\mxx$.  An object $v \in \mxx(T)$ is formally versal at every closed point $t \in T$  if and only if it is formally smooth over $\mxx$. 

\end{theorem}

\begin{proof} 
Let $T=\spe R$ and assume that $v$ is formally versal at every closed point $t \in T$ (the other implication is trivial). Consider the lifting problem in \eqref{into: defn formal smoothness} for the infinitesimal extension $A' \to A, \ \on{ker}(A' \to A)=M$ and a given map $f: R \to A$. We may assume that $M^2=0$ and that $M$ is a finitely generated $A_0$-module.




Since $D$ is compatible with completion, we can apply Theorem \ref{versality} to conclude that $T$ satisfies the lifting condition in \eqref{nei} \'etale locally around each closed point $t \in T$;  apply Theorem \ref{versality} to the henselizations of $A,A'$ and $R$ at each closed point in $T$ and use the fact that $\mxx$ is limit preserving, see Remark \ref{liftingremark}. 

We therefore have \'etale covers,
 $e_{0,i}: A_0 \to \Pi C_i$,
 $e_i': A' \to \Pi B_i'$ and $e_i: A \to \Pi B_i$ 
and a family of dotted arrows in $\mxx$  completing the diagrams
\begin{equation}
    \begin{tikzcd}
    v \arrow[rr, dashed] \arrow[rd] & & b_i' \arrow[dl]\\
    & b_i & \\
    \end{tikzcd}
\end{equation}
where $b_i=e_i^*(a)$, $b_i'= (e_i')^*(a')$. We also have an \'etale cover $e_{0,i}: A_0 \to \Pi C_i$ with $b_{0,i}=(e_{0,i}^*)(a_0)$.

The obstruction to the existence of an arrow completing the diagram 
\begin{equation}\label{liftr}
\begin{tikzcd} 
R \arrow[r, dashed] \arrow[rd, "f"] & A' \arrow[d] \\
                                        & A           
\end{tikzcd}
\end{equation}
is a class $\frak{o}$ in $\on{Ext}^1(\phi_*L_R^{\bullet}, M)_{ev}$, which is local for the \'etale topology. Thus $\frak{o}=0$, and so there is some arrow $f': R \to A'$ determined up to a derivation $d \in \derp(R, M)$. As in the proof of Theorem \ref{versality}, we have reduced the problem to proving that the map $c: \derp(R, M) \to D_{a_0}(M)$ is surjective. 

We can repeat the proof of Theorem \ref{versality} and find that the map $\derp(R, M \otimes_{A_0} C_i) \to D_{b_{0,i}}(M \otimes_{A_0} C_i)$ is surjective for each $i$. 
Both terms are compatible with \'etale localization and so we get a surjection $\derp(R, M) \otimes_{A_0} C_i \to D_{a_0}(M) \otimes_{A_0} C_i$ by (A6).  Thus, the map $\derp(R , M) \to D_{a_0}(M)$ is surjective because it is surjective after \'etale base change. \footnote{ More precisely, for each $t \in \spe A_0$, 
we have $\on{Der}(R, M)_{t} \otimes_{(A_0)_t} C_i \to D(M)_t \otimes_{(A_0)_t} C_i $ is surjective, where $C_i$ is an \'etale $(A_0)_t$-superalgebra, and $\derp(R,M)_t, D(M)_t$ is the localization at the maximal ideal $\mfr_t$ of $t$. Thus $\on{Der}(R,M)_t \to D(M)_t$ is surjective since it is surjective after \'etale base change. Therefore, $\on{Der}(R, M) \to D(M)$ is surjective since it is surjective after localization at  every maximal ideal.   }

\end{proof}

\begin{appendices}

\section{Proof of Schlessinger's theorem} 
\label{schlessinger proof appendix}



Here we follow \cite{schlessinger1968functors} very closely (except in the proof of Lemma \ref{surjcot}). 

\begin{lemma}[Compare {\cite[Lemma 1.1]{schlessinger1968functors}}] \label{surjcot}
A morphism $B \to A$ of local, complete, Noetherian superalgebras is surjective  if and only if the induced map on Zariski cotangent spaces $\tanb \to \tana$ is surjective. 

\end{lemma}


\begin{proof} For the forwards implication,  note that the map $B/\frk{B}^2 \to A/\faka^2$ induced by $B \to A$ is surjective, and that
\[ B/\fmb^2 = k \oplus \fmb/\fmb^2, \ \ \text{and} \ \ A/\faka^2 = k \oplus \faka/\faka^2 \] 
as super vector spaces. 
The map $B/\fmb^2 \to A/\faka^2$ clearly descends to a surjection $\tanb=\fmb/\fmb^2 \to \faka/\faka^2=\tana$ of super vector spaces.  

For the converse, set $A_n=A/\fmb^{n}$, $B_n=B/\fmb^n$ and consider the inverse systems $\{ A_n, f_{m,n}: A_m \to A_n \}$ and $\{ B_n, g_{m,n}: B_m \to B_n \} $ for $m \ge n$. Let $\{ \phi_n: B_n \to A_n \} $ be the morphism of inverse systems induced by $B \to A$. 
We will now prove that $\phi_n$ is surjective for each $n$. 
For $n=2$, $\phi_n$ is surjective by assumption. 
Apply $\on{Sym}_k^{n-1}(-)$ 
to the surjection $\fmb/\fmb^2 \to \faka/\faka^2$. The resulting map 
$$ \on{Sym}_k^{n-1}(\fmb/\fmb^2) \to \on{Sym}_k^{n-1}(\faka/\faka^2)$$ is surjective by right-exactness of $\on{Sym}_k^{n-1}(-)$. 
The maps $\on{Sym}_k^{n-1}(\fmb/\fmb^2) \to \fmb^{n-1}/\fmb^n$ and $\on{Sym}_k^{n-1}(\faka/\faka^2) \to \faka^{n-1}/\faka^n$ are clearly surjective.  The surjectivity of $\phi_n: \fmb^{n-1}/\fmb^n \to \faka^{n+1}/\faka^n$ now follows from commutativitity of the diagram  
\begin{center}
    \begin{tikzcd}
     \on{Sym}_k^{n-1}(\fmb/\fmb^2)  \arrow[r] \arrow[d] & \on{Sym}_k^{n-1}(\faka/\faka^2) \arrow[d] \\ 
     \fmb^{n-1}/\fmb^n \arrow[r, "\phi_n"] & \faka^{n+1}/\faka^n
    \end{tikzcd}
\end{center}

Observe that $\varprojlim \{\phi_n \}_n$ converges to $B \to A$ since $A$ and $B$ are complete. Thus to conclude that $B \to A$ is surjective it suffices to note that the inverse system $\{A_n, f_{m,n} \}$ satisfies the Mittag-Leffler (ML) condition. Indeed, $\{A_n, f_{m,n} \}$ is an inverse system of finite dimensional super vector spaces and is thus easily seen to satisfy ML  (see e.g, \cite[Example 9.1.3]{hartshorne2013algebraic}).

\end{proof}

Henceforth, let $p: B \to A$ be a surjection of local, Artin, superalgebras. Recall that we call $p$ a small extension if the kernel of $p$ is a non-zero principal ideal such that $\fmb t = (0)$. 






\begin{definition}[Compare {\cite[Definition 1.3]{schlessinger1968functors}}]
A small extension $p: A \to B$ of local, Artin superalgebras is called essential if for any other morphism $q: C \to A$ of local, Artin superalgebras with $pq$  surjective implies that $q$ is surjective. 
\end{definition}

\begin{lemma}[Compare {\cite[Lemma 1.4]{schlessinger1968functors}}] \label{essentials}

Let $p: A \to B$ be a small extension of local, Artin superalgebras. Then

\begin{itemize}
    \item[(a)] $p$ is essential if and only if the induced map $p_*: \tas \to \bas$ on Zariski cotangent spaces is an isomorphism. 
    
    \item[(b)]$p$ is not essential if and only if $p$ has a section $s: B \to A$ such that $ps = \on{id}_B$. 
\end{itemize}

\end{lemma}

\begin{proof} 

(a). Suppose $p_*$ is an isomorphism and let $q:C \to A$ be any other map with $C \to A \to B$ surjective. Then the induced map  $t_C^* \to \tas \lgr \bas$ is surjective, by Lemma \ref{surjcot}. In particular, $q_*: t_C^* \to \tas$ is surjective since $\tas \cong \bas$. Thus, $q$ is surjective, and $p$ is essential.

Conversely, suppose that $p$ is essential. Since $p$ is surjective $p_*: \tas \to \bas$ is surjective and $\on{sdim}_k \bas \le \on{sdim}_k \tas$.

Let $t_1, \dots, t_r, \alpha_1, \dots, \alpha_{r'}, |t_i|=0, |\alpha_i|=1$ be a basis for  $\bas$ and let $\ov{t}_1, \dots, \ov{t}_r, \ov{\alpha}_1, \dots, \ov{\alpha}_s \in A$ be a choice of lift of each $t_i, \alpha_i$ to $A$. Let
\[ C= k[\ov{t}_1, \dots, \ov{t}_r, \ov{\alpha}_1, \dots, \ov{\alpha}_{r'}] \subseteq A. \]
be the sub-superalgebra of $A$ generated by these lifts and consider the associated map $C \hr A \to B$. Since $t_C^* \cong \bas$, $C \hr A \to B$ is surjective. Thus $C \hr A$ is surjective, since $p$ is is essential. 
This means that $\on{sdim}_k \tas \le \on{sdim}_k t_C^*=\on{sdim}_k \bas$ and so $\on{sdim}_k t_A^*= \on{sdim}_k t_B^*$. In particular, $p_*: \tas \to \bas$ is an isomorphism since it is a surjection of super vector space of the same dimension. 
\newline

(b). Suppose $p$ has a section $s: B \to A$ and observe that if $p$ were essential, $s$ would have to be surjective, which is impossible.   

Conversely, suppose $p$ is not essential. By part (a), we have a strict inequality $\on{sdim}_k \bas < \on{sdim}_k \tas$ and so $C \subset A$ is strictly contained in $A$. Recall that $p$ maps $C$ onto $B$. We will now prove that $C \to B$ is also injective as this will automatically give us a section. 
Let $J = \on{ker}(C \hr A \to  B)$ and suppose that $J$ is non-zero. Observe that either $J=(t)$ or $J \subset (t)$ where $(t) = \on{ker}(A \to B$). 

Suppose $J=(t)$. We then have the following morphism of small extensions of $B$ by $(t)$, 
\begin{center}
    \begin{tikzcd}
    0 \arrow[r] & (t) \arrow[r] \arrow[d,no head,double] & C \arrow[r, "\phi"] \arrow[d, hook] & B \arrow[d, no head, double] \arrow[r] & 0 \\
    0 \arrow[r] & (t) \arrow[r] & A \arrow[r] & B \arrow[r] & 0
    \end{tikzcd}
\end{center}
where $\phi: C \to B$ is the composition $C \hr A \overset{p}{\to} B$. Since any morphism of small extensions is an isomorphism, we must have that $C=A$. However,  $C$ is strictly contained in $A$ and we have arrived at a contradiction. Thus, we are left to consider the case $J \subset (t)$.

If $J \subset (t)$, then $J= (ct)$ for some non-zero $c$ in $k$. However, every non-zero element in $k$ is invertible and so $J=(ct) =(t)$, which we already saw is impossible. Thus $J = (0)$ and $C = B$. 

\end{proof}

We arrive now at the super version 
Schlessinger's theorem. 
Fix $x \in \mxx(\spe k)$ and restrict the functor $\omxx_x$ to the category $C$ of Artin superschemes. 

\begin{theorem}[Compare 
{\cite[Theorem 2.11]{schlessinger1968functors}}] Suppose \emph{(S1)} and \emph{(S2)} hold for $\omxx_x$. Then there exists a complete, local, noetherian superalgebra $A$ and a formally versal, formal object $\whxi \in \wh{\omxx}_x(\on{Spf} A)$. 
\end{theorem}

\begin{proof} Set $\omxx:= \omxx_x$.  The goal is to  construct simultaneously a formal superscheme $\spf A$ and a formal object $\whxi $ over $\spf A$ such that $\whxi$ has the infinitesimal lifting property in \eqref{formallyversallift}.

We recall that $\tant$ is a finite dimensional super vector space by (S1)b and (S2) and let $\on{dim}_k \plutan = r $ and $\on{dim}_k \negtan = s$. Choose a dual basis $x_1, \dots, x_r, \theta_1, \dots, \theta_s$ for $\tant$ and let
\[ S = k[[x_1, \dots, x_r, \theta_1, \dots, \theta_s]], \ \ |x_i|=0, |\theta_i|=1,  \]

We will now prove that $\whxi$ and $\spf A$ can be constructed as the projective limit of successive quotients of $S$. Let
\begin{itemize}
    \item $A_1= S/\frk{S}\cong k$
    \item $A_2=S/\frk{S}^2 S$
\end{itemize}
and observe that
\begin{equation} \label{atwo} \spe(A_2) \cong \spe(\dual) \sqcup_k \cdots \sqcup_k \spe(\dual) \sqcup_k \spe(\dodd) \sqcup_k \cdots \sqcup_k \spe(\dodd), \ (r \text{-times}, s  \text{-times})  \footnote{To convince yourself that such an isomorphism exists, consider $S=k[[x, \theta]]$. Elements of $S/\mfr_S^2$ are of the form $c_0 + c_1 x + c_2 \theta$ where $c_0, c_1, c_2 \in k$. Elements of $\dual \times_k k[\eta]$ are of the form $(k_0 + k_1 \epsilon, k_0 + k_1' \eta)$ where $k_0, k_1, k_1' \in k$. Define $S/\mfr^2 \lgr \dual \times_k k[\eta]: x \mapsto \epsilon, \theta \mapsto \eta$ }  \end{equation}  
where $\dual$ are the ordinary dual numbers and $|\eta|=1$. Recall that the set $\on{Hom}(\spe \sdual, \spe A_2)$ is the Zariski tangent space $t_{A_2}$ of $\spe A_2$. Let $\bigsqcup_{i=1}^r \spe( \dual)_i$ and $\bigsqcup_{i=1}^s \spe(k[\eta])_i$ denote the factors in \eqref{atwo}. Then, since $\on{Hom}$ is grading preserving, we have that 
\[ t_{A_2}= \on{Hom}( \spe(\dual), \bigsqcup_{i=1}^r \spe( \dual)_i) \oplus  \on{Hom}_k(\spe(\dodd),\bigsqcup_{i=1}^s \spe(k[\eta])_i )  \] 
is a decomposition of $t_{A_2}$ into its even and odd components. 

We will now prove that there exists an object $\whxi_2$ in $\omxx$ over $\spe A_2$ inducing a bijection $t_{A_2} \lgr \tant$.  By (S1)b and \eqref{atwo}, 
\begin{equation} \label{tangentspace} \omxx( \spe A_2) = (\plutan)^r \times (\negtan)^s. \end{equation} 
Now let $X_1, \dots, X_r, \Theta_1, \dots, \Theta_s $ be a basis for $\tant$ and set
\[ \whxi_2 = (X_1, \dots, X_r, \Theta_1, \dots, \Theta_s).  \] 
Observe that $\whxi_2$ is an object in $\omxx(\spe A_2)= (\plutan)^r \times (\negtan)^s$, that we have a natural map \[ t_{A_2}  \to \tant: (d: \spe \sdual \to \spe A_2) \to  d^* \whxi_2, \]  
and that the inclusions
\begin{equation} \label{canbasis} \{ p_i: \spe(\dual) \hr \bigsqcup_{i=1}^r \spe(\dual)_i \}_i, \ \ \{q_i: \spe(\dodd) \hr \bigsqcup_{i=1}^s \spe(k[\eta])_i \}_i \end{equation} 
form a basis for $t_{A_2}$. Observe that $(p_i)^* \whxi_2=X_i$ and $(q_i)^* \whxi_2 = \Theta_i. $
Therefore, $t_{A_2} \to \tant$ is an isomorphism since 
any morphism $\sdual \to \spe A_2$ is linear combination of the inclusions $p_i$ and $q_i$ and since the $X_i$ and $\Theta_i$ form a basis for $\tant$.

We now proceed by induction. 
Suppose we have found a pair $(\whxi_q,A_q=S/J_q)$ formally versal up to order $q$. We claim that there exists an ideal $J_{q+1} \subset S$ which is minimal amongst ideals $J$ satisfying the following conditions
\begin{itemize}
    \item[(a)] $\frk{S}J_q \subseteq J \subseteq J_q$
    \item[(b)] $\whxi_q$ lifts to $S/J$. \footnote{At first glance it seems like we should take $J_{q+1}=\frk{S}^{q+1}$; However, $\frk{S}^{q+1}$ may not satisfy condition (b) and so we need to be more careful.}
\end{itemize}

Let $\mathscr{J}$ denote the set of ideals in $S$ satisfying the above conditions and note that each $J \in \msj$ is a sub super vector space of $J_q/\frk{S}J_q$. Thus it suffices to prove that $\msj$ is stable under pairwise intersection, that is if $J,K \in \msj$ then $J \cap K \in \msj$. Clearly, given any $J,K \in \msj$ we may enlarge $J$ so that $J + K = J_q$ without changing the intersection of $J \cap K$. From this we obtain that
\[ S/J \times_{S/J_q} S/K \cong S/(J \cap K). \]
From condition (S1)a we have a surjection 
\[ \mxx(S/(J \cap K))= \mxx(S/J \times_{S/J_q} S/K) \to \mxx(S/J) \times_{\mxx(S/J_q)} \mxx(S/K)  \] 
from which it follows that $\whxi_q$ lifts to $S/J \cap K$ (since $v_q$ lifts to $S/J$ and $S/K$). We now define $J_{q+1}$ to be the intersection of all members in $\msj$ and take $A_{q+1}=S/J_{q+1}$ and let $v_{q+1} \in \mxx(A_{q+1})$ be any extension of $\whxi_q$ to $A_{q+1}$. 
Now define 
\[ J = \bigcap_{q=2,3, \cdots} J_q \]
and set $A=S/J = \varprojlim_q S/J_q$, and $\whxi = \varprojlim \whxi_j \in \whxx(\spf A)$. 

We now claim that the pair $(\whxi, A)$ is formally versal. Let $(B',b') \to (B,b) \to (k,x)$ be an infinitesimal extension with kernel $I$ (a finite dimensional $k$ super vector space) where $b', b$ and $x$ are the respective objects in $\mxx$. Suppose we are given a map $u =p(\whxi \to b): A \to B$. Since $\frk{B} \cdot I=0$, we get an isomorphism
\[ B' \times_B B' \lgr B' \times_k k[I]. \] 
Now applying condition (S1).b, we get a surjection
\[ \mxx_b(B') \times (\tant \otimes I) \to \mxx_b(B') \times \mxx_b(B')  \] 
which says that the space $\tant \otimes I$ acts transitively on $\mxx_b(B')$. Therefore, if we can show that there exists a map $ A \to B'$ compatible with the given map $u: A \to B$, then the problem reduces to proving that there is a surjection between the set of maps $A \to B'$ compatible with $u:A \to B$ and $\tant \otimes I$.  

We first show that there exists a map $A \to B'$ compatible with $u: A \to B$. Since $A = \varprojlim_q S/J_q$, we can factor $u$ as $(A, \whxi) \to (A_q, \whxi_q) \to (B, b)$ for some $q$. Thus it suffices to prove that there exists an arrow completing the diagram 
\begin{center}
\begin{tikzcd}
A_{q+1} \arrow[d] \arrow[r, dashed] & B' \arrow[d] \\
A_q \arrow[r]                       & B           
\end{tikzcd}
\end{center}
This is equivalent to the existence of an arrow filling in the diagram 
\begin{center}
\begin{tikzcd}
S \arrow[d] \arrow[r, "w"]                & A_q \times_B B' \arrow[d,"p"] \arrow[r] & B' \arrow[d] \\
A_{q+1} \arrow[r] \arrow[ru, "u'", dashed] & A_q \arrow[r]                       & B           
\end{tikzcd}
\end{center}
where $w$ was chosen to make the diagram commute. To prove that $\whxi$ exists we will now show that the map $p$ has a section. Suppose $p$ does not have a section, then by Lemma \ref{essentials} $p$ is essential and $w$ is a surjection. Thus, $A_q \times_B B' = S/\on{ker}(w)$. Furthermore, from condition (S1)a, applied to $A_q \times_B B'$, we find that $\whxi_q \in \mxx(A_q)$ must lift back to an object in $\mxx(A_q \times_B B') = \mxx(S/\on{ker}(w))$. By construction, this means that $\on{ker}(w) \subseteq J_{q+1}$. Thus $w$ factors through $A_{q+1}$, and so we obtain the map $u': A \to B'$ compatible with $u: A \to B$. 

To finish the proof, recall that the set of maps $A \to B'$ compatible with $A \to B$ is a torsor over $\on{Der}_k(A, I)= t_A \otimes I$ where $I=\on{ker}(B' \to B)$. We already proved that $t_A \cong \tant$ and so $t_A \otimes_k I \cong \tant \otimes I$.

\end{proof}

\section{Proof of Flenner's criterion} 
\label{flenner proof appendix}

Here we prove Theorem \ref{flennerm}, 
following \cite{flenner1981kriterium} except
in the proof of Lemma \ref{copr}, where we can
reduced to the bosonic case rather than recapitulating 
the original proof. 

\subsection{\normalsize Half-Exact functors}

\begin{definition} We say that a functor $F: \on{Coh}(R) \to \on{Coh}(R)$ is \emph{half-exact} if for every exact sequence
\[ 0 \to M' \to M \to M'' \to 0 \]
in $\on{Coh}(R)$, the sequence
\[ F(M') \to F(M) \to F(M'')  \]
is exact. 
\end{definition} 


\begin{lemma}[Compare to {\cite[Lemma 2.1]{flenner1981kriterium}}] \label{halfexact} Suppose \emph{(S1)a} holds. Then the functor $D_{v_0}(-)$ is half-exact. 
\end{lemma}

\begin{proof} 

Let
\begin{equation} \label{fgmod} 0 \to M' \overset{f}{\to} M \overset{g}{\to} M''  \to 0 \end{equation} be an exact sequence of of finitely generated modules of a reduced ring $R_0$.  Let $v_0 \in \mxx(\spe R_0)$ and consider the sequence
\[ \dvnot(M') \overset{j}{\to} \dvnot(M) \overset{p}{\to} \dvnot(M''): v_0 \mapsto f^* v_0 \mapsto g^*(f^*v_0).  \]

We will first prove that $p \circ j=0$, \emph{i.e.}, that there is a section $s: \spe R_0[M''] \to \spe R_0$ with an arrow $g^*(f^* v) \to v_0$ in $\mxx$ over it.  The sequence in \eqref{fgmod} induces a sequence of trivial extensions of $R_0$, 
 \begin{equation} \label{pushy} R_0[M'] \to R_0[M] \to  R_0[M'']: (r,m') \mapsto (r, f(m')) \mapsto (r, gf(m')) = (r,0) \end{equation}  
sending $R_0[M']$ to a copy of $R_0$ sitting inside of $R_0[M'']$. In particular, \eqref{pushy} determines a section $s: R_0 \hr R_0[M'']$. 
The maps in \eqref{pushy} induce maps $\spe R_0[M''] \overset{s}{\to} \spe R_0 \to \spe R_0[M] \to \spe R_0[M']$. In particular, the pullback of $\vnot$ to $\spe \rmpp$ factors through $\spe R_0$ and so there is an arrow $g^*(f^* v) \to v_0 $ in $\mxx$ over $s$. 



We will now prove that $p^{-1}(0)=\on{im}(j)$. Take
 $v $ be an object in $ \dvnot(M)$ such that $p(v)=0$. Recall that $0$ is the trivial extension of $v_0$ over $\spe R_0[M'']$. Consider the diagram 
 \begin{center}
\begin{tikzcd}
0 \arrow[d] \arrow[r] & v \arrow[d, dashed]        \\
v_0 \arrow[r, dashed]                           & b
\end{tikzcd}
\end{center}
in $\mxx$ over $\spe R_0 \to \spe R_0[M] \leftarrow \spe R_0[M'']$. By (S1)a, there exists an object $b$ over $\spe(R_0 \times_{R_0[M'']} R_0[M])$ filling in the dotted arrows. Since $R_0 \times_{R_0[M'']} R_0[M]=\rmp$, the object $b$ is in $\mxx(\spe \rmp)$. Thus, $v \in \on{im}(j)$.  
\newline

\end{proof}

\begin{remark} \label{halfexactext} Lemma \ref{halfexact} holds  for the functor
 $\ex(v,-): \on{Coh}(R) \to \on{Set}$ in the following sense. For every exact sequence of finitely generated $R$-modules, 
 $ 0 \to M' \to M \to M'' \to 0 $, the sequence 
\begin{center}
    \begin{tikzcd}
    \ex(v,M') \arrow[r, "j"] & \ex(v, M) \arrow[r, "p"] & \ex(v, M'') 
    \end{tikzcd}
\end{center} 
is such that $p^{-1}(0)=\on{im}(j)$. 
\end{remark} 


\subsection{\normalsize Vergleichsatz (Comparison Theorem) }

In this section we closely follow \cite[Section 6]{flenner1981kriterium}. 

\begin{lemma}[Compare with {\cite[Lemma 6.2]{flenner1981kriterium}}]  Let $R$ be a Noetherian superalgebra and let \[ \tilde{R} =R[T_1, \dots, T_k, T_{k+1}, \dots, T_{k+s}], \ \  |T_i|=0, \  1 \le i \le r, \ \  |T_{k+i}|=1, \   1 \le i \le s, \]. Let $M$ be a ($\z_{ \ge 0} \times \z_2$)-graded $\rtwid$-module. Then

\begin{itemize}
    \item[\emph{(i)}] Let 
    \[ M' \to M \to M''\]
    be an exact sequence of graded $\rtwid$-modules. Then $M$ is finitely generated  over $\rtwid$  if
     $M', M''$ are finitely generated over $\rtwid$, 
    
    \item[(ii)] If $M/T_i M$ is a finitely generated $\rtwid$-module, then so is $M$. 

\end{itemize}

\end{lemma}

\begin{proof} Part (i) is obvious. For part (ii), suppose, by contradiction, that $M$ has infinite rank.
Then multiplication by $T_i$ must annihilate an infinite subset of the set of  generators of $M$.  Therefore, $T_i M$  must be finitely generated. Now apply part (i) to the sequence
\[ T_k M \to M \to M/T_k M, \] 
to conclude that $M$ is finitely generated over $\rtwid$

\end{proof}

Now let $\tura \subset R$ be an ideal and let $M_i, i \in \z$ be finitely generated $R$-modules. Define \[ M := \bigsqcup_{i \ge 0} M_i = M_0 \oplus M_1 \oplus \cdots . \]

Let $\jj \subset R $ denote the ideal of odd nilpotents and consider the following graded $R$-algebras, 
\begin{align*}
B_{\tura}(R) & := \bigsqcup_{i \ge 0} \tura^i R= R \oplus \tura \oplus \tura^2 \oplus \tura^3 \oplus \cdots  \\
B_{\tura}(\jj) & := \bigsqcup_{i \ge 0} \jj \tura^i = \jj \oplus \jj \tura \oplus \cdots  \\
B_{\tura}(R)/\jj B_{\tura}(R) & := \bigsqcup_{i \ge 0} \tura^i/\jj \tura^i= R/\jj \oplus \tura/\jj \tura \oplus \cdots 
\end{align*}

\begin{lemma} \label{fint} Let $R$ be a Noetherian superalgebra and let $M$ be an $R$-module. Then $M$ is finitely generated over $R$ if and only if $M_{\bos}=M/\jj M$ is finitely generated over $R_{\bos}=R/\jj$.

\end{lemma}

\begin{lemma} \label{copr} Let $F: \on{Coh}(R) \to \on{Coh}(R)$ be a half-exact functor. If $\bigsqcup_{i \ge 0} M_i$ is finitely generated over $B_{\tura}(R)$, then $F(M)= \bigsqcup_{i \in \z} F(M_i)$ is finitely generated over $B_{\tura}(R)$. 
\end{lemma}

\begin{proof} The bosonic case is proved in \cite[Satz 6.1]{flenner1981kriterium}. Let
\[ M/\jj M = \bigsqcup_{i \ge 0} M_i/\jj M_i. \] 
Observe that $M/\jj M$ is a finitely generated $B_{\tura}(R)/\jj B_{\tura}(R)$-module by Lemma \ref{fint} and that $F(M/\jj M)$ is a finitely-generated $B_{\tura}(R)/\jj B_{\tura}(R)$-module by \cite[Satz 6.1]{flenner1981kriterium}. Furthermore, $F(M/\jj M)$ is a finitely-generated $ B_{\tura}(R)$-module since the map $B_{\tura}(R) \to B_{\tura}(R)/\jj B_{\tura}(R)$ is surjective. 

Now apply the functor $F$ to the exact sequence of $B_{\tura}(R)$-modules
\[ 0 \to \jj M \to M \to M/\jj M \to 0\] 
to get an exact sequence of $B_{\tura}(R)$-modules
\[ F(\jj M) \to F(M) \to F(M/\jj M). \]
We now claim that $F(\jj M)$ is finite. We want to prove this claim because it automatically implies that $F(M)$ is finite; from the above sequence we get an injective map $F(M)/F(\jj M) \to F(M/\jj M)$, with $F(M/\mm J)$ is finitely generated, and so if $F(\jj M)$ is finite so must be $F(M)$, otherwise the quotient $F(M)/F(\jj M)$ would be infinite and we would have a contradiction.

We prove the claim by induction on $\jj$. Let $s$ be the smallest integer for which $\jj^{s+1}=0$.
We can apply \cite[Satz 6.1]{flenner1981kriterium} to the finite $B_{\tura}(R)/\jj B_{\tura}(R)$-module $\jj M/\jj^{2} M$ to conclude that $F(\jj M/\jj^{2} M)$ is finite over $B_{\tura}(R)/\jj B_{\tura}(R)$ and thereby also over $B_{\tura}(R)$. 
Now $\jj M/\jj^{3} M$ is finite over $\jj M/\jj^2 M$ and thus also over $B_{\tura}(R)/\jj B_{\tura}(R)$. Thus $F(\jj M/\jj^{3} M)$ is finite over $B_{\tura}(R)/\jj B_{\tura}(R)$ and thereby also over $B_{\tura}(R)$. By induction on $k=1, \dots, s$ we find that $F(\jj M)$ is finite over $B_{\tura}(R)/\jj B_{\tura}(R)$ and thereby also over $B_{\tura}(R)$. 
Thus $F(M)$ is finite and this concludes the proof.


\end{proof}


\begin{corollary}[Compare with {\cite[Corollary 6.3]{flenner1981kriterium}}] \label{indy} Let $F: \on{Coh}(R) \to \on{Coh}(R)$ be a half-exact functor and let $\tura \subset R $ be an ideal.  Then

\begin{itemize}

\item[(1)] $\bigsqcup_{i \ge 0} F(\tura^i M)$ is a finite $\bigsqcup_{i \ge 0} \tura^i$-module, and 

\item[(2)] The canonical map
\[  \{ F(M)/\tura^i F(M)  \}_{i \ge 0} \to \{ F(M/\tura^i M) \}_{i \ge 0} \]
is injective. 
\end{itemize} 

\end{corollary} 

\begin{proof} (2) follow from Lemma \ref{copr}, because $\bigsqcup \tura^i M$ is a finite $\bigsqcup \tura^i R$-module

\end{proof}

\begin{corollary}[Compare with {\cite[Corollary 6.4]{flenner1981kriterium}} ] \label{extid}
The module $\bigsqcup \on{Ext}^i(L_{R}^{\bullet}, \tura^n M)$ is a coherent $\bigsqcup \tura^n$-modules. In particular, the natural maps
\[\phi_n: \on{Ext}^i(L_{R}^{\bullet}, M)/\tura^n \on{Ext}^i(L_{R}^{\bullet}, M) \to \on{Ext}^i(L_{R}^{\bullet}, M/ \tura^n M)  \] 
are bijective for all $n$. 
\end{corollary}

\begin{lemma}[Compare with {\cite[Lemma 3.1]{flenner1981kriterium}}] \label{extcon}  Let $R$ be a local, noetherian $k$-superalgebra and let $v \in \mxx(T=\spe R)$. Let $\frak{a} \subseteq R$ be an ideal and let $M$ be a finitely generated $R_0$-module. Let $p$ denote the canonical map 
\[ p: \ex(v, M) \to \varprojlim_n \ex(v, M/\frak{a}^n M). \] 
Then $p^{-1}(0)=0$.

\end{lemma}

\begin{proof}

Let the subscript $n$ denote the residue class modulo $\frak{a}^n$ 
and consider the diagram: 
\begin{center}
\begin{tikzcd}
{\derp(R, M)} \arrow[r, "c"] \arrow[d] & D_v(M) \arrow[r, "i"] \arrow[d, "p"] & {\ex(v, M)} \arrow[r, "\pi"] \arrow[d] & {\ex(T, M)} \arrow[d] \\
{\derp(R, M_n)} \arrow[r, "c_n"]           & D_v(M_n) \arrow[r, "i_n"]           & {\ex(v, M_n)} \arrow[r, "\pi_n"]           & {\ex(T, M_n)}          
\end{tikzcd}
\end{center}

We begin by proving that the natural map $q:\ex(T,M) \to \varprojlim \ex(T,M_n)$ is injective. We have an identification of $\ex(T,-)$ with the half-exact functor $T_R^1(-)_{ev} = \on{Ext}(L_R^{\bullet}, -)_{ev}$.  Applying Corollary \ref{extid} we obtain a bijection
\[ \on{Ex}(T,M)_n = \ex(T,M_n) \]
for all $n$. Consider the commutative diagram
\begin{center}
\begin{tikzcd}
{\ex(T,M)} \arrow[d, "q"] \arrow[rrd, two heads, "q_n"] &                                             &              \\
{\varprojlim \ex(T,M_n)} \arrow[r]          & {\ex(T,M_n)} \arrow[r, no head, Rightarrow,] & {\ex(T,M)_n}. 
\end{tikzcd}
\end{center} 
Suppose $f \in \ex(T,M)$ such that $q(f)=0$. Then $q_n(f)=0$ for all $n$ and so $f \in \bigcap \tura^n \ex(T,M)=0$, by the Krull intersection theorem. In particular, we find that 
\[ \on{ker}(q \circ \pi) = \on{ker}(\pi).  \] 
Thus if $g \in \ex(v, M)$ is such that $p(g)=0$, then $g \in \pi^{-1}(0)$. This reduces the problem to proving that the map
\[ \pi^{-1}(0) \to \varprojlim \pi_n^{-1}(0) \] 
is injective. Using the identification between $\pi^{-1}(0)$ and $\on{coker}(c)$, this is equivalent to proving that the map
\[ \on{coker}(c) \to \varprojlim \on{coker}(c_n) \]
is injective. 

Since $\derp(T,-) = \on{Ext}^0(L_{R}^{\bullet},-)_{ev}$, we can apply Corollary \ref{extid}  to conclude that the map 
\[ \derp(T,M)_n \to \derp(T, M_m) \]
is bijective. Since $D_v(-)$ is half-exact, we can apply Corollary \ref{indy} to conclude that the map
\[ D_v(M)_n \to D_v(M_n) \]
is injective. 
The fact that $\on{coker}(c) \to \varprojlim \on{coker}(c_n)$ is injective the follows from a standard diagram chase. 








\end{proof}

\subsection{\normalsize Formal versality: Vanishing of the extension functor}

We now arrive at Flenner's criterion.

\begin{theorem} Let $v$ be an object in $\mxx$ over an affine superscheme $T$.   Assume that \emph{(S1)} and \emph{(S2)} are satisfied. Then the following statements are equivalent. 
\begin{enumerate}
    \item[(1)] $v$ is formally versal, 
    \item[(2)] $\exc(v, k(t))_t=\exc(v, \Pi k(t))_t=0$. 
    \item[(3)] $\exc(v, \mm)_t=0$ for every coherent $\otr$-module $\mm$. 
\end{enumerate}
\end{theorem}

 \begin{proof}

Fix $t \in T$ and set $\exc(v, -)=\exc(v, -)_t$. Since everything is affine, we can go ahead and replace $\mm$ with its associated module $M$, $\ot$ with $R$, etc. Let $\mfr=\mfr_t \subset R$ denote the maximal ideal of $t$.  

We will first prove that (2) and (3) are equivalent. That condition (3) implies (2) is obvious since $k(t)$ is a finitely generated $\ot$-module, so assume that $\exc(v, k(t))_t=\exc(v, \Pi k(t))_t=0$.

From Lemma  \ref{extcon} we know that $\exc(v, M)=0$ if and only if $\varprojlim \exc(v, M/\mfr^n)=0$. Thus it suffices to prove that $\exc(v, k(t))=\exc(v, \Pi k(t))=0$  implies that  $\exc(v, M/\mfr^n)=0$ for all $n$. 
For each $n \ge 1$, we have an exact sequence
 \begin{equation} \label{infrr} 0 \to \mfr^n/\mfr^{n+1} \to M/\mfr^{n+1} \to M/\mfr^n \to 0.  \end{equation} 
 and by Remark \ref{halfexactext} a sequence 
 \[ \exc(v, \mfr^n/\mfr^{n+1}) \overset{j}{\to} \exc(v, M/\mfr^{n+1} ) \overset{p}{\to} \exc(v, M/\mfr^n)  \]
 of pointed sets such that $p^{-1}(0)=\on{im}(j)$. Observe that $\exc(v, M/\mfr^{n+1})=0$ if and only if $\exc(v, M/\mfr^n/\mfr^{n+1})=\exc(v, M/\maxnp)=0$. It thus suffices to prove that if $\exc(v, k(t))=\exc(v, \Pi k(t))=0$, then $\exc(v, \mfr^n/\mfr^{n+1})=\exc(v, M/\maxnp)=0$. 
 
When $n=1$,  $\mfr/\mfr^2$ and  $M/\mfr$ are finite dimensional $k(t)$ super vector spaces. 
Any extension of $R$ by a finite dimensional super vector space can be factored into a finite composite of  extensions by $k(t)$ and $\Pi k(t)$.
Thus the implication holds for $n=1$. The sequence  \eqref{infrr} now shows that this result holds for all $n$ since $\maxn/\maxnp$ are finite dimensional $k(t)$ super vector spaces for all $n$.



We will now prove that (1) implies (2). 
Let $i:T \hr T'$ be an extension and let $v \to v'$ be an arrow in $\mxx$ over $i$. Let $T_n=\spe R/\mfr^n$,  $T_n'= \spe R'/\mfr^n$, and $i_n: T_n \hr T_n', \  M(n)=\on{ker}(R_n' \to R_n)$. 
Let $v_n  = v \times_T T_n$ and $v_n' = v' \times_T T_n'$. 

Since $v$ is formally versal, there exist an arrow completing the diagrams
\begin{center}
 (i). \begin{tikzcd} 
T_n \arrow[d, hook] \arrow[r, "i_n"] & T_n' \arrow[d] \arrow[ld, dashed] \\
T \arrow[r, "i"]                    & T'                               
\end{tikzcd} \ \ \ \ (ii). \begin{tikzcd}
v_n \arrow[d, hook] \arrow[r] & v_n' \arrow[d, dashed] \arrow[ld, dashed] \\
v \arrow[r, dashed]                   & v'                               
\end{tikzcd}  
\end{center}
The diagram on the left is cocartesian and so there exists a section $\alpha: T' \to T$. Thus, $T'$ is trivial extension of $T$ by $M(n)$.

We will now prove that $v'$ is the pushout $v \times_{v_n} v_n'$ of the diagram (ii) by showing there is an arrow $v \times_{v_n} v_n' \to v'$  By (S1)a, the natural map \begin{equation} \label{cot} \omxx(T') \to \omxx_{v_n}(T_n') \times \omxx_{v_n}(T).  \end{equation}
is a surjection. 
In particular, the pair $(v_n', v ) \in \omxx_{v_n}(T_n') \times \omxx_{v_n}(T')$ to an object over $T'$. We know that $v' \in \omxx(T')$ is an example of such a lift. Since $T_n'$ is the trivial extension of $T_n$ 
it follows from condition (S1)b that $v'$ is the unique lift of $v_n' \times_{v_n} v$ to $T'$. Thus $v'$ is a pushout of diagram (ii). 
Since $v'$ is the pushout, there exists a section $v' \to v $ which implies that $v' = v \times_T T'$. Thus $\exc(v, k(t))=\exc(v, \Pi k(t)) =0$. 

Now suppose that $\exc(v, k(t))=\exc(v, \Pi k(t))=0$. Let  $i:X \hr X'$ be an extension, let $b \in \mxx(X)$ and let $X \to T$ be a morphism of superschemes. We can represent this data with the diagrams
\begin{center}
(i). \begin{tikzcd}
X \arrow[d, "i"', hook] \arrow[r] & T \\
X'             &  
\end{tikzcd}           
\ \ \ (ii). 
\begin{tikzcd}
b \arrow[r] \arrow[d] & v \\
b' & {}
\end{tikzcd}
\end{center} 
where the map $b \to b'$ is over $i$. The goal is to prove that there exists a map $X' \to T$ and a map $b' \to v$ lying over it. 
Consider the diagrams 
\begin{center}
(i). \begin{tikzcd}
X \arrow[d, "i"', hook] \arrow[r] & T \arrow[d, dashed] \\
X' \arrow[r, dashed] &  T \sqcup_X X' 
\end{tikzcd}           
\ \ \ (ii). 
\begin{tikzcd}
b \arrow[r] \arrow[d] & v \arrow[d, dashed] \\
b' \arrow[r, dashed] &  v'
\end{tikzcd}
\end{center} 
Here $T \sqcup_X X'$ is an extension of $T$ by $k(t)$ or $\Pi k(t)$. Since $\exc(v, k(t))=\exc(v, \Pi k(t))=0$, $T \sqcup_X X'$ is necessarily the trivial extension of $T$ and so there exists a section $T \sqcup_X X' \to T$. In particular, we have a map $X' \to T \sqcup_X X' \to T$.

From condition (S1), we have a surjection
\[ \omxx_b(T \sqcup_X X') \to \omxx_b(X') \times \omxx_b(T)\] which ensures that there exists an object $v' \in \omxx_b(T \sqcup_X X')$ completing the diagram on the right. Since $\exc(v, k(t))=\exc(v, \Pi k(t))=0$, we must have that $v'$ is a trivial extension of $v$ to $T \sqcup_X X'$. There, therefore, exists a section $v' \to v$ from which we obtain a map $b' \to v' \to v'$ lying over $X' \to T \sqcup_X X' \to T$. 
\end{proof}

\end{appendices}

\bibliographystyle{amsalpha}
\bibliography{References} 

\end{document}